\documentclass{amsart}

\usepackage{amssymb,amsmath,latexsym,amsthm}
\usepackage[english]{babel}

\newtheorem{thm}{Theorem}[section]
\newtheorem{theorem}[thm]{Theorem}
\newtheorem{lemma}[thm]{Lemma}

\theoremstyle{definition}
\newtheorem{remark}[thm]{Remark}

\numberwithin{equation}{section}

\usepackage{mathtools}
\usepackage{microtype}  
\usepackage{url}

\newcommand{\eps}{\varepsilon}

\newcommand{\dx}{\mathrm{d}} 
\newcommand{\E}{\mathcal{E}} 
\newcommand{\A}{\mathcal{A}} 
\newcommand{\B}{\mathcal{B}} 
\newcommand{\I}{\mathcal{I}} 
\newcommand{\J}{\mathcal{J}} 

\newcommand{\N}{\mathbb{N}}
\newcommand{\R}{\mathbb{R}} 

\newcommand{\second}{{\prime\prime}}

\newcommand{\Etilde}{\widetilde{E}}

\newcommand{\Stilde}{\widetilde{S}}
 
\newcommand{\Vtilde}{\widetilde{V}}

\newcommand{\Odi}[1]{\Odip{}{#1}}
\newcommand{\Odig}[1]{\mathcal{O}\Bigl(#1\Bigr)}
\newcommand{\Odim}[1]{\mathcal{O}\bigl(#1\bigr)}   
\newcommand{\Odip}[2]{\mathcal{O}_{#1}\left(#2\right)}
\newcommand{\Odipg}[2]{\mathcal{O}_{#1}\Bigl(#2\Bigr)}  
\newcommand{\Odipm}[2]{\mathcal{O}_{#1} (#2)}  
\newcommand{\odip}[2]{{o}_{#1}\left(#2\right)}
\newcommand{\odi}[1]{\odip{}{#1}}
\newcommand{\odim}[1]{o\bigl(#1\bigr)}

\makeatletter
\newenvironment{Biggcases}{%
  \matrix@check\Biggcases\env@Biggcases
}{%
  \endarray %
}
\def\env@Biggcases{%
  \let\@ifnextchar\new@ifnextchar
  \Biggl\lbrace
  \def\arraystretch{1.2}%
  \array{@{}l@{\quad}l@{}}%
}
\makeatother

\begin{document}

\title[Asymptotic formulae for binary problems]{Short intervals asymptotic formulae for binary problems with prime   powers} 


\author{Alessandro Languasco}
\address{Alessandro Languasco\\
Universit\`a di Padova,  Dipartimento di Matematica  ``Tullio Levi-Civita'', Via Trieste 63, 35121 Padova \\
Italy}
\email{alessandro.languasco@unipd.it}

\author{Alessandro Zaccagnini}
\address{Alessandro Zaccagnini\\
Universit\`a di Parma, Dipartimento di Scienze Matematiche,  Fisiche e Informatiche, Parco Area delle Scienze, 53/a, 43124 Parma \\
Italy}
\email{alessandro.zaccagnini@unipr.it  }




\maketitle
%
%

\begin{abstract}
We prove  results about the
asymptotic formulae in short intervals 
for the average number  of representations
of  integers of the forms $n=p_{1}^{\ell_1}+p_{2}^{\ell_2}$,
with $\ell_1, \ell_2\in\{2,3\}$, $\ell_1+\ell_2\le 5$ are fixed integers, 
and  $n=p^{\ell_1} + m^{\ell_2}$, with $\ell_1=2$ and $2\le \ell_2\le 11$ 
or $\ell_1=3$ and $ \ell_2=2$  are fixed integers,
 $p,p_1,p_2$ are prime numbers and $m$ is an integer.
\end{abstract}

\bigskip
\section{Introduction}
Let $N$ be a sufficiently large integer and $1\le H \le N$.
In our recent papers   \cite{LanguascoZ2016b} and \cite{LanguascoZ2017a}
we provided   suitable asymptotic formulae in short intervals for the  
number of representation of an integer $n$ as a sum of a prime and 
a prime square, as a sum of a prime and a square, as the sum of two prime squares
or as a sum of a prime square and a square. 

In this paper we generalise the approach already used there to look
for the asymptotic formulae for more difficult binary problems. 
To be able to formulate or statements in a precise way we need more
definitions. 
Let $\ell_1, \ell_2\ge 1$ be integers, 
\begin{equation}
\label{density-def-c-def}
\lambda: =\frac{1}{\ell_1}+\frac{1}{\ell_2}\le 1 \quad \textrm{and}\quad
c(\ell_1,\ell_2):=
\frac{\Gamma(1/\ell_1)\Gamma(1/\ell_2)}{\ell_1\ell_2\Gamma(\lambda)}
=
c(\ell_2,\ell_1).
\end{equation}  
Using these notations we can say that our results in \cite{LanguascoZ2016b} and \cite{LanguascoZ2017a}
are about $\lambda=3/2$ and $\lambda=1$ while here we are interested in the case $\lambda\le 1$.
We also recall that   Suzuki \cite{Suzuki2017b,Suzuki2017} has recently sharpened
our results in   \cite{LanguascoZ2017a} for the case $\lambda=3/2$.

Finally  let
\begin{equation}
\label{A-def} 
A=A(N,d) :=
 \exp \Bigl( d   \Bigl( \frac{\log N}{\log \log N} \Bigr)^{\frac{1}{3}} \Bigr),
\end{equation} 
where $d$ is a real parameter (positive or negative) chosen according
to need, and
\[
\sum_{n=N+1}^{N+H} 
r^{\second}_{\ell_1,\ell_2}(n),
\quad
\textrm{where}
\quad 
r^{\second}_{\ell_1,\ell_2}(n) = \!\!\!\!\!\!\!\!
\sum_{\substack {p_{1}^{\ell_1}+p_{2}^{\ell_2}=n\\N/A \le  p_{1}^{\ell_1} , \, p_{2}^{\ell_2} \le  N}}
\!\!\!\!\!\!\!\!
\log p_{1} \log p_{2}.
\]
 
Due to the available estimates on primes in almost all short intervals and due
to  $\lambda\le 1$, we are unconditionally able
to get a non-trivial result only for  
$\ell_1, \ell_2\in\{2,3\}$, $\ell_1+\ell_2\le 5$;
in fact, since for this additive problem
we can interchange the role of the prime powers involved,
such a  condition is equivalent to   $\ell_1=2, \ell_2\in\{2,3\}$.
\begin{theorem}
\label{thm-uncond}
 Let $N\ge 2$, $1\le H \le N$ be integers.
Moreover let  $\ell_1=2, \ell_2\in\{2,3\}$. 
Then, for every $\eps>0$, there exists $C=C(\eps)>0$ such that
\begin{align*} 
\sum_{n=N+1}^{N+H} &
r^{\second}_{2,\ell_2}(n)
=
c(2,\ell_2)
H  N^{\lambda-1} 
+
\Odipg{\ell_2}{ H N^{\lambda-1} A(N, -C(\eps)) },
\end{align*}
uniformly for    $N^{\frac32-\frac{11}{6\ell_2} +\eps}\le H \le N^{1-\eps}$,
where $\lambda$ and $c(2,\ell_2)$ are defined in \eqref{density-def-c-def}. 
\end{theorem}

Clearly for $\ell_2=2$ Theorem \ref{thm-uncond} coincides with the result proved in \cite{LanguascoZ2016b},
but for $\ell_2=3$ it is new.

Assuming the Riemann Hypothesis (RH)  holds and taking
\begin{equation}
\label{r-def}
R^{\second}_{\ell_1,\ell_2}(n)= 
\sum_ {p_{1}^{\ell_1}+p_{2}^{\ell_2}=n }
\log p_{1} \log p_{2},
\end{equation}
we get a non-trivial result for    $\sum_{n = N+1}^{N + H} 
R^{\second}_{\ell_1,\ell_2}(n)$ uniformly for every $2\le \ell_1 \le \ell_2$
and $H$ in some range. Let further
\begin{equation}
\label{aB-def}
a(\ell_1,\ell_2):=\frac{\ell_1}{2(\ell_1-1)\ell_2}\in \Bigl(0, \frac12\Bigr]
\quad
\text{and}
\quad
b(\ell_1):=\frac{3\ell_1}{2(\ell_1-1)}\in \Bigl(\frac32,3\Bigr].
\end{equation}
We use throughout the paper the convenient notation $f=\infty(g)$
for $g=\odi{f}$.

\begin{theorem}
\label{thm-RH}
 Let $N\ge 2$, $1\le H \le N$,  $2\le \ell_1\le \ell_2$ be integers
 and assume the Riemann Hypothesis holds.
Then
\begin{equation*} 
   \sum_{n = N+1}^{N + H} 
R^{\second}_{\ell_1,\ell_2}(n)
=
c(\ell_1,\ell_2)HN^{\lambda-1}
+\Odipg{\ell_1,\ell_2}{
H^2N^{\lambda-2}
+
H^{\frac{1}{\ell_1}} N^{\frac{1}{2\ell_2}}( \log N)^{\frac{3}{2}}
}
\end{equation*}
uniformly for  
 $\infty(N^{1-a(\ell_1,\ell_2)}( \log N)^{b(\ell_1)})\le H \le\odi{N}$,
 where $\lambda$ and $c(\ell_1,\ell_2)$ are defined in 
 \eqref{density-def-c-def}, $a(\ell_1,\ell_2), b(\ell_1)$ are defined in \eqref{aB-def}.
\end{theorem}

Clearly for $\ell_1=\ell_2=2$, Theorem \ref{thm-RH} coincides with the result proved in \cite{LanguascoZ2016b}
but in all the other cases it is new.
To prove Theorem \ref{thm-RH} we will have to use the original
Hardy-Littlewood generating functions to exploit the wider 
uniformity over $H$ they allow;
see the remark after Lemma \ref{LP-Lemma-gen}.

A slightly different problem is the one in which we replace
a prime power with a power. Letting
 \[
 r^{\prime}_{\ell_1,\ell_2}(n) = \!\!\!\!\!\!\!\!
\sum_{\substack {p^{\ell_1}+m^{\ell_2}=n\\N/A \le  p^{\ell_1} , \, m^{\ell_2} \le  N}}
\!\!\!\!\!\!\!\!
\log p, 
\]
we have the following
\begin{theorem}
\label{thm-uncond-HL}
 Let $N\ge 2$, $1\le H \le N$. Moreover let $\ell_1, \ell_2\ge2$.
Then, for every $\eps>0$, there exists $C=C(\eps)>0$ such that
\begin{align*} 
\sum_{n=N+1}^{N+H} &
r^{\prime}_{\ell_1,\ell_2}(n)
=
c(\ell_1,\ell_2)H  N^{\lambda-1} 
+
\Odipg{\ell_1,\ell_2}{H N^{\lambda-1} A(N, -C(\eps)) },
\end{align*}
uniformly for  $N^{2-\frac{11}{6\ell_1} -\frac{1}{\ell_2}+\eps}\le H \le N^{1-\eps}$
with $\ell_1=2$ and $2\le \ell_2\le 11$, or $\ell_1=3$ and $ \ell_2=2$,
where $\lambda$ and $c(\ell_1,\ell_2)$ are defined in \eqref{density-def-c-def}. 
\end{theorem}

Clearly for $\ell_1=\ell_2=2$, Theorem \ref{thm-uncond-HL} coincides with the result proved in \cite{LanguascoZ2016b}
but in all the other cases it is new.
In this case we cannot interchange the role of the prime powers as we can do for
the first two theorems we proved; hence the  different  condition on the size of $H$.  
 
In the conditional case, as for the proof of Theorem  \ref{thm-RH},
we need to use the Hardy-Littlewood original functions, but in this case 
we are forced to restrict our analysis to the $p^\ell+m^2$ problem
due to the lack of an analogue of the functional equation
\eqref{omega-approx} in the general case. 
It is well known that this is crucial in these problems.
Letting
 \[
R^{\prime}_{\ell,2}(n)= 
\sum_ {p^{\ell}+m^2=n }
\log p,
\]
we have the following
\begin{theorem}
\label{thm-RH-HL}
 Let $N\ge 2$, $1\le H \le N$,  $\ell\ge 2$ be integers and
assume the Riemann Hypothesis  holds. Then
\begin{align*} 
   \sum_{n = N+1}^{N + H} 
R^{\prime}_{\ell,2}(n)
&=
 c(\ell,2)  HN^{\frac{1}{\ell}-\frac{1}{2}} 
 \\
&\hskip1cm + \Odipg{\ell}{\frac{H^2}{N^{\frac{3}{2}-\frac{1}{\ell}}}+ \frac{H N^{\frac{1}{\ell}-\frac{1}{2}}\log \log N} {(\log N)^{\frac{1}{2}}} 
+
H^{\frac{1}{2}}N^{\frac{1}{2\ell}} \log N}
\end{align*}
uniformly for  
$\infty(N^{1-\frac{1}{\ell}}(\log N)^{2}) \le H \le \odi{N}$,
where $ c(\ell,2) $ is defined in \eqref{density-def-c-def}.
\end{theorem}

Clearly for $\ell =2$, Theorem \ref{thm-RH-HL} coincides with the result proved in \cite{LanguascoZ2016b}
but in all the other cases it is new.
The proof of Theorem \ref{thm-RH-HL} needs the use of the functional equation \eqref{omega-approx}
and hence it is different from the one of Theorem \ref{thm-RH}.

We finally remark that we deal with a similar problem with a $k$-th power of a prime and two squares of primes in 
\cite{LZ2018}.

\textbf{Acknowledgement}.
We thank the anonymous referee for his/her precise remarks.

\section{Setting}
Let $\ell \ge 2$, $\ell_1, \ell_2\ge 2$ be  integers,  $e(\alpha) = e^{2\pi i\alpha}$,
$\alpha \in [-1/2,1/2]$,
\begin{align}
\notag
S_{\ell}(\alpha) &= \sum_{N/A \le  m^{\ell} \le  N}\Lambda(m)\ e(m^{\ell} \alpha),   \quad
V_{\ell}(\alpha) = \sum_{N/A \le  p^{\ell} \le  N} \log p\ e(p^{\ell} \alpha), \\
\label{main-defs}
T_{\ell}(\alpha) &= \sum_{N/A \le m^{\ell} \le  N} e(m^{\ell} \alpha ), \ 
f_{\ell}(\alpha)  =\frac{1}{\ell}\sum_{N/A \le  m\le  N} m^{\frac{1}{\ell}-1}\ e(m\alpha), \\ 
\notag
U(\alpha,H) &= \sum_{1\le  m\le  H}e(m\alpha),
\end{align}
where $A$ is defined in \eqref{A-def}.
 We also have the usual numerically explicit inequality
\begin{equation}
\label{UH-estim}
\vert U(\alpha,H) \vert
\le
\min \bigl(H;  \vert \alpha\vert ^{-1}\bigr),
\end{equation}
see, \emph{e.g.}, on page 39 of Montgomery \cite{Montgomery1994},
and, by Lemmas 2.8 and 4.1 of Vaughan \cite{Vaughan1997}, we obtain
\begin{equation}
\label{f-ell-T-ell-estim}
f_{\ell}(\alpha) \ll_\ell 
\min \bigl(N^{\frac{1}{\ell}};  \vert \alpha\vert ^{-\frac{1}{\ell}}\bigr); \quad
\vert T_{\ell}(\alpha) - f_{\ell}(\alpha) \vert \ll (1+\vert \alpha \vert N)^{\frac12}.
\end{equation}
Recalling that $\eps>0$,  we let $L= \log N$  and
\begin{equation}
\label{B-def}
B=B(N,c,\ell_1,\ell_2)=
N ^{1-\lambda} A(N, c),
\end{equation} 
where  $\lambda$ is defined in 
 \eqref{density-def-c-def} and $c=c(\eps)>0$ will be chosen later.

\section{Lemmas}

 \begin{lemma} 
\label{UH-average}
Let $H\ge 2$, $\mu\in \R$, $\mu\ge 1$. Then
\[
\int_{-\frac12}^{\frac12} 
\vert U(\alpha,H)\vert^{\mu} \, \dx \alpha 
\ll
\begin{cases}
\log H& \text{if}\ \mu=1 \\
H^{\mu-1}& \text{if}\ \mu>1.
\end{cases}
\]
\end{lemma}
\begin{proof}
By \eqref{UH-estim} we can write that
\[
\int_{-\frac12}^{\frac12} 
\vert U(\alpha,H)\vert^{\mu} \, \dx \alpha 
\ll
H^\mu \int_{-\frac1H}^{\frac1H}  \dx \alpha +
\int_{\frac1H}^{\frac12} 
\frac{\dx \alpha}{\alpha^\mu}
\]
and the result follows immediately.
\end{proof}

\begin{lemma}
\label{trivial-lemma}
Let $\ell >0$ be a real number. Then
\(
\vert S_{\ell}(\alpha)- V_{\ell}(\alpha) \vert 
\ll_{\ell}
 N^{\frac{1}{2\ell}}   .
\)
\end{lemma}
\begin{proof}
Clearly we have
\begin{align*}
 \vert S_{\ell}(\alpha)- V_{\ell}(\alpha) \vert 
\le 
\sum_{k= 2}^{\Odi{L}}\sum_{p^{k\ell}\le  N}   \log p\ \ll_\ell
\int_{2}^{\Odi{L}}    N^{1/(t\ell)}\  \dx t
\ll_\ell N^{\frac{1}{2\ell}},
 \end{align*}
where in the last but one inequality we used  a weak form of the Prime Number Theorem.
\end{proof}

We need the following lemma which collects the results 
of Theorems 3.1-3.2 of \cite{LanguascoZ2013b}; see also Lemma 1 of \cite{LanguascoZ2016a}.
\begin{lemma}  
\label{App-BCP-Gallagher}
Let $\ell  > 0$ be a real number and $\eps$ be an arbitrarily small
positive constant. Then there exists a positive constant 
$c_1 = c_{1}(\eps)$, which does not depend on $\ell$, such that
\[
\int_{-\frac{1}{K}}^{\frac{1}{K}}
\vert
S_{\ell}(\alpha) - T_{\ell}(\alpha)  
\vert^2 
\, \dx \alpha
\ll_{\ell}   N^{\frac{2}{\ell}-1}
\Bigl(
A(N, - c_{1})
+
\frac{K L^{2}}{N}
\Bigr),
\]
uniformly for  $N^{1-\frac{5}{6\ell}+\eps}\le K \le N$.  
Assuming further RH we get 
\[
\int_{-\frac{1}{K}}^{\frac{1}{K}}
\vert
S_{\ell}(\alpha) - T_{\ell}(\alpha) 
\vert^2 
\, \dx \alpha
\ll_{\ell} 
\frac{N^{\frac{1}{\ell}} L^{2}}{K} + K N^{\frac{2}{\ell}-2} L^{2},
\]
uniformly for  $N^{1-\frac{1}{\ell}}\le K \le N$.  
\end{lemma}
 
 Combining the two previous lemmas we get
\begin{lemma}  
\label{App-BCP-Gallagher-2}
Let $\ell  > 0$ be a real number and $\eps$ be an arbitrarily small
positive constant. Then there exists a positive constant 
$c_1 = c_{1}(\eps)$, which does not depend on $\ell$, such that
\[
\int_{-\frac{1}{K}}^{\frac{1}{K}}
\vert
V_{\ell}(\alpha) - T_{\ell}(\alpha)  
\vert^2 
\, \dx \alpha
\ll_{\ell}   N^{\frac{2}{\ell}-1}
\Bigl(
A(N, - c_{1})
+
\frac{K L^{2}}{N}
\Bigr),
\]
uniformly for  $N^{1-\frac{5}{6\ell}+\eps}\le K \le N$.  
Assuming further RH we get 
\[
\int_{-\frac{1}{K}}^{\frac{1}{K}}
\vert
V_{\ell}(\alpha) - T_{\ell}(\alpha) 
\vert^2 
\, \dx \alpha
\ll_{\ell} 
\frac{N^{\frac{1}{\ell}} L^{2}}{K} + K N^{\frac{2}{\ell}-2} L^{2},
\]
uniformly for  $N^{1-\frac{1}{\ell}}\le K \le N$.  
\end{lemma}
\begin{proof}
By Lemma \ref{trivial-lemma} 
we have that 
\[
\int_{-\frac{1}{K}}^{\frac{1}{K}}
\vert
S_{\ell}(\alpha) - V_{\ell}(\alpha) 
\vert^2 
\, \dx \alpha
\ll_\ell
\frac{N^{\frac{1}{\ell}}}{K}
\]
and the result follows
using the inequality $\vert a + b \vert^2 \le 2 \vert a  \vert^2 + 2\vert b \vert^2 $
and   Lemma \ref{App-BCP-Gallagher}.
\end{proof}

\begin{lemma} 
\label{zac-lemma}
Let $\ell\ge 2$ be an integer and $0<\xi\le  \frac{1}{2}$. Then
\[
\int_{-\xi}^{\xi} \vert T_{\ell}(\alpha) \vert ^2\, \dx \alpha 
\ll_\ell
 \xi N^{\frac{1}{\ell}} + 
\begin{cases}
L & \text{if}\ \ell =2\\
1 & \text{if}\ \ell > 2,
\end{cases}
\]
\[
\int_{-\xi}^{\xi} \vert S_{\ell}(\alpha) \vert ^2\, \dx \alpha 
\ll_\ell
N^{\frac{1}{\ell}} \xi L +
\begin{cases}
L^{2}& \text{if}\ \ell =2\\
1 & \text{if}\ \ell > 2
\end{cases}
\]
and
\[
\int_{-\xi}^{\xi} \vert V_{\ell}(\alpha) \vert ^2\, \dx \alpha 
\ll_\ell
N^{\frac{1}{\ell}} \xi L +
\begin{cases}
L^{2}& \text{if}\ \ell =2\\
1 & \text{if}\ \ell > 2.
\end{cases}
\]
\end{lemma}
\begin{proof}
The first two parts were proved in Lemma 1 of \cite{LanguascoZ2017a}.
Let's see the third part.
By symmetry we can  integrate over $[0,\xi]$.
We use Corollary 2 of Montgomery-Vaughan  \cite{MontgomeryV1974}  
with $T=\xi$, $a_r= \log r$ if $r$ is prime,  $a_r= 0$ otherwise 
and $\lambda_r= 2\pi r^\ell$ thus getting
\begin{align*}
\int_{0}^{\xi} \vert V_{\ell}(\alpha) \vert ^2\, \dx \alpha 
&= \!\!\!\!
\sum_{N/A \le r^{\ell} \le N} a(r)^2 \bigl(\xi +\Odim{\delta_r^{-1}}\bigr)
\ll_{\ell}
N^{\frac{1}{\ell}} \xi L +
\sum_{p^{\ell} \le N}  (\log p)^2  p^{1-\ell},
\end{align*}
since $\delta_{r} = \lambda_r - \lambda_{r-1} \gg_{\ell} r^{\ell-1}$. 
The last error term is $\ll_{\ell}1$ if $\ell >2$ and $\ll L^2$ otherwise. %
The third part of Lemma \ref{zac-lemma} follows. 
\end{proof}

 \begin{lemma}
 \label{media-f-ell}
Let $\ell >0$ be a real number and recall that $A$ is defined in \eqref{A-def}. Then
\[
\int_{-\frac{1}{2}}^{\frac{1}{2}} \vert f_{\ell}(\alpha) \vert ^2\, \dx \alpha 
\ll_{\ell} 
N^{\frac{2}{\ell}-1}
\begin{cases}
A^{1-\frac{2}{\ell}} & \text{if}\ \ell > 2\\
\log A & \text{if}\ \ell = 2\\
1 & \text{if}\ 0<\ell < 2.
\end{cases}
\]
 \end{lemma}
 \begin{proof}
 By Parseval's theorem we have
 \[
 \int_{-\frac12}^{\frac12} \vert f_{\ell}(\alpha) \vert ^2\, \dx \alpha 
 = 
\frac{1}{\ell^2} \sum_{N/A \le m \le N} m^{\frac{2}{\ell}-2}
 \]
and the lemma follows at once.
 \end{proof}
 
 We also need similar lemmas for the Hardy-Littlewood functions since, in the conditional case,
 we will use them. Let
\begin{equation*} 
\Stilde_\ell(\alpha) = \sum_{n=1}^{\infty} \Lambda(n) e^{-n^{\ell}/N} e(n^{\ell}\alpha),  \quad
\Vtilde_\ell(\alpha) = \sum_{p=2}^{\infty} \log p \, e^{-p^{\ell}/N} e(p^{\ell}\alpha),
\end{equation*} 
\textrm{and}
\begin{equation*} 
z= 1/N-2\pi i\alpha.
\end{equation*} 
We remark that 
\begin{equation}
\label{z-estim}
\vert z\vert ^{-1} \ll \min \bigl(N, \vert \alpha \vert^{-1}\bigr).
\end{equation}
 \begin{lemma}[Lemma 3 of  \cite{LanguascoZ2016b}]
\label{tilde-trivial-lemma}
Let $\ell\ge 1$ be an integer. Then
\(
\vert \Stilde_{\ell}(\alpha)- \Vtilde_{\ell}(\alpha) \vert 
\ll_{\ell}
 N^{\frac{1}{2\ell}}  .
\)
\end{lemma}

\begin{lemma}[Lemma 2 of  \cite{LanguascoZ2016a}]
Let $\ell \ge 1$ be an integer, $N \ge 2$  and $\alpha\in [-1/2,1/2]$.
Then
\begin{equation*}
\Stilde_{\ell}(\alpha)  
= 
\frac{\Gamma(1/\ell)}{\ell z^{\frac{1}{\ell}}}
- 
\frac{1}{\ell}\sum_{\rho}z^{-\frac{\rho}{\ell}}\Gamma\Bigl(\frac{\rho}{\ell}\Bigr) 
+
\Odip{\ell}{1},
\end{equation*}
where $\rho=\beta+i\gamma$ runs over
the non-trivial zeros of $\zeta(s)$.
\end{lemma}
\begin{proof}
It follows the line of  Lemma 2 of  \cite{LanguascoZ2016a}; we just 
correct an oversight in its  proof. In eq. (5) on page 48 of 
\cite{LanguascoZ2016a} the term 
\(
-  
  \sum_{m=1}^{\ell \sqrt{3}/4} \Gamma (- 2m/\ell ) z^{2m/\ell}
\)
is missing. Its estimate is trivially $\ll_{\ell} \vert z \vert^{\sqrt{3}/2} \ll_{\ell} 1$.
Hence such an oversight does not affect the final result of  
Lemma 2 of  \cite{LanguascoZ2016a}.
\end{proof}

\begin{lemma} [Lemma 4 of \cite{LanguascoZ2016a}]
 \label{Laplace-formula}
Let $N$ be a positive integer, $z=1/N-2\pi i \alpha$, $\alpha\in [-1/2,1/2]$, and 
$\mu > 0$.
Then
\[
  \int_{-1 / 2}^{1 / 2} z^{-\mu} e(-n \alpha) \, \dx \alpha
  =
  e^{- n / N} \frac{n^{\mu - 1}}{\Gamma(\mu)}
  +
  \Odipg{\mu}{\frac{1}{n}},
\]
uniformly for $n \ge 1$.
\end{lemma}

\begin{lemma}[Lemma 3 of \cite{LanguascoZ2016a} and 
Lemma 1 of \cite{LanguascoZ2016b}]
 \label{LP-Lemma-gen} 
Let $\eps$ be an arbitrarily small
positive constant,  $\ell \ge 1$ be an integer, $N$ be a
sufficiently large integer and $L= \log N$. Then there exists a positive constant 
$c_1 = c_{1}(\eps)$, which does not depend on $\ell$, such that 
\[
\int_{-\xi}^{\xi} \,
\Bigl\vert
\Stilde_\ell(\alpha) - \frac{\Gamma(1/\ell)}{\ell z^{\frac{1}{\ell}}}
\Bigr\vert^{2}
\dx \alpha 
\ll_{\ell}
 N^{\frac{2}{\ell}-1} A(N, - c_{1}) 
\]
uniformly for $ 0\le \xi < N^{-1 +5/(6\ell) - \eps}$.
Assuming RH we get 
\[
\int_{-\xi}^{\xi} \,
\Bigl\vert
\Stilde_\ell(\alpha) - \frac{\Gamma(1/\ell)}{\ell z^{\frac{1}{\ell}}}
\Bigr\vert^{2}
\dx \alpha 
\ll_{\ell}
N^{\frac{1}{\ell}}\xi L^{2}
\]
uniformly  for  $0 \le \xi \le \frac{1}{2}$.
\end{lemma} 
\begin{proof}
It follows the line of Lemma 3 of \cite{LanguascoZ2016a} and 
Lemma 1 of \cite{LanguascoZ2016b}; we just 
correct an oversight in their  proofs. Both eq. (8) on page 49 of 
\cite{LanguascoZ2016a}  and eq. (6) on page 423
of \cite{LanguascoZ2016b} should read as 
\[
\int_{1/N}^{\xi}
\Big \vert\sum_{\rho\colon \gamma > 0}z^{-\rho/\ell}\Gamma (\rho/\ell ) \Big \vert^2 
\dx \alpha 
\le
\sum_{k=1}^K
\int_\eta^{2\eta} \Big \vert\sum_{\rho\colon \gamma > 0}z^{-\rho/\ell}\Gamma (\rho/\ell ) \Big \vert^2 \dx \alpha,
\]
where $\eta=\eta_k= \xi/2^k$, $1/N\le \eta \le \xi/2$  and $K$ is a suitable integer satisfying $K=\Odi{L}$. 
The remaining part of the proofs are left untouched. 
Hence such  oversights do not affect the final result of  
Lemma 3 of \cite{LanguascoZ2016a} and 
Lemma 1 of \cite{LanguascoZ2016b}.
\end{proof}
\begin{remark}
The main difference between  Lemma \ref{LP-Lemma-gen}
and Lemma \ref{App-BCP-Gallagher-2} is the larger uniformity 
over $\xi$ in the conditional estimate.
Hence, under the assumption of RH, 
Lemma \ref{LP-Lemma-gen} will allow us to avoid
the unit interval splitting (see \eqref{dissect}
below). This will lead to milder conditions
on $H$  than something like 
$N^{1-\frac{1}{\ell_1}}B\le H \le N$ which Lemma \ref{App-BCP-Gallagher-2} would require
in the conditional analogue of
 \eqref{E1-estim}, for example.
 In conclusion, in the conditional case  
  Lemma \ref{LP-Lemma-gen} will give us a wider
 $H$ and $(\ell_1,\ell_2)$ ranges, while,
 unconditionally,  Lemma \ref{LP-Lemma-gen}
 and Lemma \ref{App-BCP-Gallagher-2}
 are essentially equivalent.
 \end{remark}

\begin{lemma}
\label{partial-int-average}
Let   $\ell \ge 1$ be an integer, $N$ be a
sufficiently large integer and $L= \log N$.
Assume RH. We have 
\[
\int_{-\frac12}^{\frac12} \,
\Bigl\vert
\Stilde_\ell(\alpha) - \frac{\Gamma(1/\ell)}{\ell z^{\frac{1}{\ell}}}
\Bigr\vert^{2}
\vert U(-\alpha,H) \vert
\ \dx \alpha 
\ll_{\ell}
N^{\frac{1}{\ell}}  L^{3}.
\] 
\end{lemma}

\begin{proof}
Let
$\Etilde_{\ell}(\alpha) : =\Stilde_\ell(\alpha) -  \Gamma(1/\ell)/(\ell z^{\frac{1}{\ell}})$.
By \eqref{UH-estim} we have 
\begin{align}
\notag
\int_{-\frac12}^{\frac12} \,
& \vert
\Etilde_{\ell}(\alpha)
 \vert^{2}
\vert U(-\alpha,H) \vert
\ \dx \alpha 
\\ \notag
&\ll
H
\int_{-\frac1H}^{\frac1H} \,
 \vert
\Etilde_{\ell}(\alpha)
 \vert^{2}
\ \dx \alpha
+
\int_{\frac1H}^{\frac12} \,
 \vert
\Etilde_{\ell}(\alpha)
 \vert^{2}
\frac{\dx \alpha}{\alpha}
+
\int_{-\frac12}^{-\frac1H} \,
 \vert
\Etilde_{\ell}(\alpha)
 \vert^{2}
\frac{\dx \alpha}{\alpha}
\\
\label{M-split}
&
=
M_1+M_2+M_3,
\end{align}
say. By Lemma \ref{LP-Lemma-gen} we immediately get that
\begin{equation}
\label{M1-estim}
M_1 \ll_{\ell} N^{\frac{1}{\ell}}  L^{2}.
\end{equation}
By a partial integration and Lemma \ref{LP-Lemma-gen} we obtain
\begin{align}
\notag
M_2 
&\ll 
\int_{-\frac12}^{\frac12} \,
 \vert
\Etilde_{\ell}(\alpha)
 \vert^{2}
\ \dx \alpha
+
H \int_{-\frac1H}^{\frac1H} \,
 \vert
\Etilde_{\ell}(\alpha)
 \vert^{2}
\ \dx \alpha 
+
\int_{\frac1H}^{\frac12}
\Bigl(
\int_{-\xi}^{\xi}
 \vert
\Etilde_{\ell}(\alpha)
 \vert^{2}
\ \dx \alpha
\Bigr)
\frac{\ \dx \xi}{\xi^2}
\\&
\label{M2-estim}
\ll_{\ell} 
N^{\frac{1}{\ell}}  L^{2}
+
\int_{\frac1H}^{\frac12} \frac{N^{\frac{1}{\ell}} \xi  L^{2}}{\xi^2}\ \dx \xi
\ll_{\ell} 
N^{\frac{1}{\ell}}  L^{3}.
\end{align}
A similar computation leads to $M_3 \ll_{\ell}  N^{\frac{1}{\ell}}  L^{3}$.
By \eqref{M-split}-\eqref{M2-estim}, the lemma follows.
\end{proof}

 \section{Proof of Theorem \ref{thm-uncond}}
By now we let $2\le \ell_1\le \ell_2$; we'll see at the end
of the proof how the conditions in the statement of this theorem
follow.
Assume $H>2B$. 
We have 
\begin{align}
\notag 
\sum_{n=N+1}^{N+H} &
r^{\second}_{\ell_1,\ell_2}(n) = 
\int_{-\frac12}^{\frac12} V_{\ell_1}(\alpha)  V_{\ell_2}(\alpha)  U(-\alpha,H)e(-N\alpha) \, \dx \alpha
\\
\notag
&=
\int_{-\frac{B}{H}}^{\frac{B}{H}} V_{\ell_1}(\alpha) V_{\ell_2}(\alpha)  U(-\alpha,H)e(-N\alpha) \, \dx \alpha 
 \\
 \label{dissect}
& \hskip1cm 
+ 
\int\limits_{I(B,H)} \!\!\!\!\!  
V_{\ell_1}(\alpha) V_{\ell_2}(\alpha)  U(-\alpha,H)e(-N\alpha) \, \dx \alpha,
%
\end{align}
where  $I(B,H):=[-1/2,-B/H]\cup  [B/H, 1/2]$.
By the Cauchy-Schwarz inequality we have
\begin{align*}
\int\limits_{I(B,H)} \!\!\! 
V_{\ell_1}(\alpha) V_{\ell_2}(\alpha)  U(-\alpha,H)e(-N\alpha) \, \dx \alpha
&
 \ll
\Bigl(
\int\limits_{I(B,H)} \!\!\! 
\vert V_{\ell_1}(\alpha) \vert ^{2} \vert U(-\alpha,H) \vert \, \dx \alpha
\Bigr)^{\frac12}
\\&
\times
\Bigl(
\int\limits_{I(B,H)} \!\!\! 
\vert V_{\ell_2}(\alpha) \vert ^{2} \vert U(-\alpha,H) \vert \, \dx \alpha
\Bigr)^{\frac12}.
\end{align*}

By \eqref{UH-estim}, Lemma \ref{zac-lemma}  
and a partial integration argument, it is clear that 
\begin{align}
\notag
\int\limits_{I(B,H)} \!\!\! 
\vert V_{\ell}(\alpha) \vert ^{2} &\vert U(-\alpha,H) \vert \, \dx \alpha
\ll
\int_{\frac{B}{H}}^{\frac12} 
\vert V_{\ell}(\alpha) \vert ^{2} \frac{\dx \alpha}{\alpha}
 \\ 
 \label{V-average}
 &
 \ll_\ell
 N^{\frac{1}{\ell}} L 
+
\frac{H L^{2}}{B} 
+
\int_{\frac{B}{H}}^{\frac12}  
(\xi N^{\frac{1}{\ell}}L  + L^{2})\ \frac{\dx \xi}{\xi^2}
 \ll_\ell
  N^{\frac{1}{\ell}} L^{2}
+
\frac{H L^{2}}{B},
\end{align}
for every $\ell\ge 2$.
Hence, recalling \eqref{B-def},
we obtain
\begin{align}
\notag
 \int\limits_{I(B,H)} \!\!\!\!
V_{\ell_1}(\alpha) V_{\ell_2}(\alpha)  U(-\alpha,H)e(-N\alpha) \, \dx \alpha
& \ll_{\ell_1,\ell_2}
  N^{\frac{\lambda}{2}} L^{2}
+
 \frac{ H^{\frac12} N^{\frac{1}{2\ell_1}}  L^{2}}{B^{\frac12}}
 +\frac{H L^{2}}{B}
 \\&
  \label{V-average-coda}
 \ll_{\ell_1,\ell_2}
\frac{H L^{2}}{B}.
\end{align}

By \eqref{dissect} and \eqref{V-average-coda} we get
\begin{align} 
\notag
\sum_{n=N+1}^{N+H} 
&r^{\second}_{\ell_1,\ell_2}(n) 
 =  
\int_{-\frac{B}{H}}^{\frac{B}{H}} V_{\ell_1}(\alpha) V_{\ell_2}(\alpha)  U(-\alpha,H)e(-N\alpha) \, \dx \alpha 
 +
\Odipg{\ell_1,\ell_2}{\frac{H L^{2}}{B}}
%
\\
\notag
&
 = \int_{-\frac{B}{H}}^{\frac{B}{H}} f_{\ell_1}(\alpha) f_{\ell_2}(\alpha)  U(-\alpha,H)e(-N\alpha) \, \dx \alpha 
 \\
 \notag
 &
\hskip0.5cm
+
\int_{-\frac{B}{H}}^{\frac{B}{H}} f_{\ell_2}(\alpha) (V_{\ell_1}(\alpha) - f_{\ell_1}(\alpha) ) U(-\alpha,H)e(-N\alpha) \, \dx \alpha 
\\
 \notag
 &
\hskip0.5cm
+
\int_{-\frac{B}{H}}^{\frac{B}{H}} f_{\ell_1}(\alpha) (V_{\ell_2}(\alpha) - f_{\ell_2}(\alpha) ) U(-\alpha,H)e(-N\alpha) \, \dx \alpha 
 \\
 \notag
 &
\hskip0.5cm
 +
\int_{-\frac{B}{H}}^{\frac{B}{H}} (V_{\ell_1}(\alpha)-f_{\ell_1}(\alpha) ) (V_{\ell_2}(\alpha) - f_{\ell_2}(\alpha) ) U(-\alpha,H)e(-N\alpha) \, \dx \alpha
 \\
\notag
 &
\hskip0.5cm+
\Odipg{\ell_1,\ell_2}{\frac{H L^{2}}{B}}
\\
 \label{main-dissection}
&= \I_1 +\I_2 + \I_3 + \I_4 + E,
 \end{align}
 say. We now evaluate these terms.
 
\subsection{Computation of the main term $\I_1$}
 \label{comp-main-term}
Recalling Definition \eqref{density-def-c-def} and that
$I(B,H)=[-1/2,-B/H]\cup  [B/H, 1/2]$, 
a direct calculation 
and \eqref{f-ell-T-ell-estim} give
\begin{align}
\notag
\I_1
&
=
\sum_{n=1}^H 
\int_{-\frac12}^{\frac12} f_{\ell_1}(\alpha) f_{\ell_2}(\alpha)   e(-(n+N)\alpha)\, \dx \alpha 
+\Odipg{\ell_1,\ell_2}{
\int\limits_{I(B,H)} \!\!\!  \frac{ \dx \alpha}{ \vert \alpha \vert^{1+\lambda}}} 
\\
\notag
&
=
\frac{1}{\ell_1\ell_2}
\sum_{n=1}^H \sum_{\substack {m_{1}+m_{2} =n+N\\ N/A \le  m_{1} \le  N \\ N/A \le  m_{2}\le  N}}
m_1^{\frac{1}{\ell_1}-1}m_2^{\frac{1}{\ell_2}-1}
+\Odipg{\ell_1,\ell_2}{  \Bigl(\frac{H}{B}\Bigr)^{\lambda}}
 \\
\label{I1-eval}
 & 
=
M_{\ell_1,\ell_2}(H,N)  +\Odipg{\ell_1,\ell_2}{  \Bigl(\frac{H}{B}\Bigr)^{\lambda}},
 \end{align}
 say.
Recalling Lemma 2.8
of Vaughan \cite{Vaughan1997} we can see that
 order of magnitude of the main term $M_{\ell_1,\ell_2}(H,N)$ is 
$ HN^{\lambda-1}$. 
We first complete the range of summation for $m_1$ and $m_2$ to the
interval $[1, N]$.
The corresponding error term is
\begin{align*}
  &\ll_{\ell_1, \ell_2} \!\!
  \sum_{n=1}^H
    \sum_{\substack {m_{1}+m_{2} =n+N\\ 1 \le  m_{1} \le  N/A \\ 1 \le  m_{2}\le  N}}
      m_1^{\frac{1}{\ell_1}-1}m_2^{\frac{1}{\ell_2}-1}
  \ll_{\ell_1, \ell_2} \!\!
  \sum_{n=1}^H \sum_{m=1}^{N/A} m^{\frac{1}{\ell_2}-1} (n+N-m)^{\frac{1}{\ell_1}-1} \\
  &\ll_{\ell_1, \ell_2}
  H N^{\frac{1}{\ell_1}-1} \sum_{m=1}^{N/A} m^{\frac{1}{\ell_2}-1}
  \ll_{\ell_1, \ell_2} \!\!
  H N^{\lambda-1} A^{- \frac{1}{\ell_2}}.
\end{align*}
We deal with the main term $M_{\ell_1,\ell_2}(H,N)$ using Lemma~2.8 of
Vaughan \cite{Vaughan1997}, which yields the $\Gamma$ factors
hidden in $c(\ell_1,\ell_2)$:
\begin{align*}
  \frac{1}{\ell_1\ell_2}
  &
  \sum_{n=1}^H \sum_{\substack {m_{1}+m_{2} =n+N\\ 1\le  m_{1} \le  N \\ 1 \le  m_{2}\le  N}}
    m_1^{\frac{1}{\ell_1}-1}m_2^{\frac{1}{\ell_2}-1}
  =
  \frac{1}{\ell_1\ell_2}
  \sum_{n=1}^H \sum_{m=1}^{N} m^{\frac{1}{\ell_2}-1}(n+N-m)^{\frac{1}{\ell_1}-1} \\
  &=
  c(\ell_1,\ell_2)
  \sum_{n=1}^H
   \Bigl[ (n+N)^{\lambda-1} +
     \Odi{(n+N)^{\frac{1}{\ell_1}-1} + N^{\frac{1}{\ell_2}-1}n^{\frac{1}{\ell_1}}}
  \Bigr] \\
  &= 
  c(\ell_1,\ell_2)
  \sum_{n=1}^H (n+N)^{\lambda-1} 
  +
  \Odipg{\ell_1,\ell_2}{H N^{\frac{1}{\ell_1}-1} +H^{\frac{1}{\ell_1}+1} N^{\frac{1}{\ell_2}-1}} \\
  &= 
  c(\ell_1,\ell_2) HN^{\lambda-1}
  +\Odipg{\ell_1,\ell_2}{ H^2N^{\lambda-2} + H   N^{\frac{1}{\ell_1}-1} +H^{\frac{1}{\ell_1}+1}   N^{\frac{1}{\ell_2}-1}}.
\end{align*}
Summing up,
\begin{align}
\notag
  M_{\ell_1,\ell_2}(H,N) 
  &=
  c(\ell_1,\ell_2) HN^{\lambda-1}
  \\&
\label{main-term}
  +\Odipg{\ell_1,\ell_2}{ H^2N^{\lambda-2} + H   N^{\frac{1}{\ell_1}-1} +H^{\frac{1}{\ell_1}+1}   N^{\frac{1}{\ell_2}-1} 
  +\frac{HN^{\lambda-1}}{ A^{\frac{1}{\ell_2}}}}.
\end{align}

\subsection{Estimate of $\I_2$}
\label{estim-I2}
Using \eqref{f-ell-T-ell-estim} we obtain
\begin{align}
\notag
\vert V_{\ell}(\alpha) - f_{\ell}(\alpha) \vert
&\le
\vert V_{\ell}(\alpha) - T_{\ell}(\alpha) \vert
+
\vert T_{\ell}(\alpha) - f_{\ell}(\alpha) \vert
\\&
\label{V-approx}
=
\vert V_{\ell}(\alpha) - T_{\ell}(\alpha) \vert
+
\Odim{(1+\vert \alpha \vert N)^{\frac12}}.
\end{align}
Hence
 \begin{align}
\notag
\I_2
& \ll
  \int_{-\frac{B}{H}}^{\frac{B}{H}} 
\vert f_{\ell_2}(\alpha) \vert \vert V_{\ell_1}(\alpha) -T_{\ell_1}(\alpha)  \vert\vert U(-\alpha,H)\vert \, \dx \alpha
\\&
\notag
\hskip1cm
 +
   \int_{-\frac{B}{H}}^{\frac{B}{H}} 
\vert  f_{\ell_2}(\alpha) \vert  (1+\vert \alpha \vert N)^{\frac12} \vert U(-\alpha,H) \vert \, \dx \alpha
\\
\label{I2-split}
&= E_1+E_2,
 \end{align}
 say.
By \eqref{UH-estim} 
we have
\begin{align*}
E_2 
&\ll
 H  \int_{-\frac{1}{N}}^{\frac{1}{N}} \vert f_{\ell_2}(\alpha)\vert \, \dx \alpha
+
H N^{\frac12} \int_{\frac{1}{N}}^{\frac{1}{H}} \vert f_{\ell_2}(\alpha)\vert \alpha^{\frac12}  \, \dx \alpha
\\& 
\hskip1cm
+
N^{\frac12}\int_{\frac{1}{H}}^{\frac{B}{H}} \vert f_{\ell_2}(\alpha)\vert\alpha^{-\frac12}  \, \dx \alpha.
\end{align*}

Hence, using the  Cauchy-Schwarz inequality and Lemma \ref{media-f-ell},
we get
\begin{align}
\notag
E_2 
&
 \ll_{ \ell_2}
 H N^{-\frac12}  \Bigl(\int_{-\frac{1}{N}}^{\frac{1}{N}} \vert f_{\ell_2}(\alpha)\vert^2\ \dx \alpha\Bigr)^{\frac{1}{2}} 
 \\&
\notag
\hskip1cm 
+
H N^{\frac12} \Bigl(\int_{\frac{1}{N}}^{\frac{1}{H}} \vert f_{\ell_2}(\alpha)\vert^2\ \dx \alpha\Bigr)^{\frac{1}{2}} 
 \Bigl(\int_{\frac{1}{N}}^{\frac{1}{H}}  \alpha   \, \dx \alpha\Bigr)^{\frac{1}{2}} 
 \\
 \notag
 &
 \hskip1cm
+
N^{\frac12}
\Bigl(\int_{\frac{1}{H}}^{\frac{B}{H}} \vert f_{\ell_2}(\alpha)\vert^2\ \dx \alpha\Bigr)^{\frac{1}{2}} 
 \Bigl(\int_{\frac{1}{H}}^{\frac{B}{H}}  \frac{\dx \alpha}{\alpha}\Bigr)^{\frac{1}{2}}  
\\ 
\notag
&
 \ll_{ \ell_2}
 \bigl(
 H  N^{\frac{1}{\ell_2}-1}  
+
  N^{\frac{1}{\ell_2}}  
+
  N^{\frac{1}{\ell_2}} L^{\frac{1}{2}}
  \bigr)  A^{\frac{1}{2}-\frac{1}{\ell_2}}  (\log A)^{1/2}
\\ 
 \label{E2-estim}
&
 \ll_{ \ell_2}
 N^{\frac{1}{\ell_2}}  A^{\frac{1}{2}-\frac{1}{\ell_2}} L^{\frac{1}{2}} (\log A)^{1/2},
\end{align}
where $A$ is defined in \eqref{A-def}.

Using  \eqref{UH-estim}, the Cauchy-Schwarz inequality,   
and Lemmas  \ref{App-BCP-Gallagher-2} and \ref{media-f-ell}
we obtain
\begin{align}
\notag
E_1&\ll 
 H 
 \Bigl(\int_{-\frac{B}{H}}^{\frac{B}{H}} \vert f_{\ell_2}(\alpha)\vert^2\ \dx \alpha\Bigr)^{\frac{1}{2}} 
  \Bigl(\int_{-\frac{B}{H}}^{\frac{B}{H}} \vert V_{\ell_1}(\alpha)-T_{\ell_1}(\alpha)\vert^2\ \dx \alpha\Bigr)^{\frac{1}{2}}  
\\ 
\label{E1-estim}  
&
\ll_{\ell_1,\ell_2}
 H \Bigl( \frac{N}{A}\Bigr)^{\frac{1}{\ell_2}-\frac{1}{2}} 
( \log A )^{1/2}
 N^{\frac{1}{\ell_1}-\frac{1}{2}}  
 A(N, -c_1)
 \ll_{\ell_1,\ell_2}
 H N^{\lambda-1}  
 A(N, -C)
\end{align}
for a suitable choice of   $C=C(\eps)>0$, 
provided that $N^{-1+\frac{\eps}{2}}<B/H<N^{-1+\frac{5}{6\ell_1}-\eps}$;
hence   $N^{1-\frac{5}{6\ell_1}+\eps}B\le H \le N^{1-\eps}$
suffices.
Summarizing,  by \eqref{density-def-c-def}, \eqref{I2-split}-\eqref{E1-estim} we obtain
that there exists $C=C(\eps)>0$ such that 
 \begin{equation}
\label{I2-estim}
\I_2 \ll_{\ell_1,\ell_2}
 H N^{\lambda-1} A(N, -C)
\end{equation}
provided that $N^{-1+\frac{\eps}{2}}<B/H<N^{-1+\frac{5}{6\ell_1}-\eps}$; 
hence   $N^{1-\frac{5}{6\ell_1}+\eps} B \le H \le N^{1-\eps}$
suffices.

\subsection{Estimate of $\I_3$}
\label{estim-I3}
It's very similar to $\I_2$'s;
we just need to interchange $\ell_1$ with $\ell_2$
thus getting that there exists $C=C(\eps)>0$ such that 
 \begin{equation}
\label{I3-estim}
\I_3 \ll_{\ell_1,\ell_2}
 H N^{\lambda-1} A(N, -C)
\end{equation}
provided that $N^{-1+\frac{\eps}{2}}<B/H<N^{-1+\frac{5}{6\ell_2}-\eps}$; 
hence   $N^{1-\frac{5}{6\ell_2}+\eps} B \le H \le N^{1-\eps}$
suffices.

\subsection{Estimate of $\I_4$}
\label{estim-I4}
By \eqref{main-dissection} and  \eqref{V-approx} we can write
 \begin{align}
 \notag
\I_4
& \ll_{\ell_1,\ell_2}
\int_{-\frac{B}{H}}^{\frac{B}{H}}
\vert V_{\ell_1}(\alpha) -T_{\ell_1}(\alpha)  \vert \vert V_{\ell_2}(\alpha) -T_{\ell_2}(\alpha)  \vert\vert U(-\alpha,H)\vert \, \dx \alpha
\\
\notag
&\hskip1cm
+
\int_{-\frac{B}{H}}^{\frac{B}{H}}
\vert  V_{\ell_1}(\alpha) -T_{\ell_1}(\alpha)\vert  (1+\vert \alpha \vert N)^{\frac12} \vert U(-\alpha,H) \vert \, \dx \alpha
\\
\notag
&\hskip1cm
+
\int_{-\frac{B}{H}}^{\frac{B}{H}}
\vert  V_{\ell_2}(\alpha) -T_{\ell_2}(\alpha)\vert  (1+\vert \alpha \vert N)^{\frac12} \vert U(-\alpha,H) \vert \, \dx \alpha
\\
\notag
&\hskip1cm
+
\int_{-\frac{B}{H}}^{\frac{B}{H}}
 (1+\vert \alpha \vert N)  \vert U(-\alpha,H) \vert \, \dx \alpha
 \\&
 \label{I4-split}
= E_3+E_4+E_5+E_6,
 \end{align}
 say.
By \eqref{density-def-c-def}, \eqref{UH-estim}, the Cauchy-Schwarz inequality and Lemma \ref{App-BCP-Gallagher-2} 
we have
\begin{align}
\notag
E_3&\ll
H
 \Bigl(
      \int_{-\frac{B}{H}}^{\frac{B}{H}}
      \vert V_{\ell_1}(\alpha) -T_{\ell_1}(\alpha)   \vert^{2} 
        \, \dx \alpha        
\Bigr)^{\frac12}
 \Bigl(
      \int_{-\frac{B}{H}}^{\frac{B}{H}}
      \vert V_{\ell_2}(\alpha) -T_{\ell_2}(\alpha)   \vert^{2} 
        \, \dx \alpha        
\Bigr)^{\frac12}  
 \\
 &
 \ll_{\ell_1,\ell_2}
 \label{E3-estim}
 H N^{\lambda-1} A(N, -C),
\end{align}
for a suitable choice of   $C=C(\eps)>0$, 
provided that $N^{-1+\frac{\eps}{2}}<B/H<N^{-1+\frac{5}{6\ell_2}-\eps}$; 
hence   $N^{1-\frac{5}{6\ell_2}+\eps} B \le H \le N^{1-\eps}$
suffices.

By \eqref{UH-estim} and the Cauchy-Schwarz inequality 
we have
\begin{align}
\notag
E_4&\ll
 H  \int_{-\frac{1}{N}}^{\frac{1}{N}} \vert V_{\ell_1}(\alpha) - T_{\ell_1}(\alpha)\vert \, \dx \alpha
+
H N^{\frac12} \int_{\frac{1}{N}}^{\frac{1}{H}} \vert V_{\ell_1}(\alpha) - T_{\ell_1}(\alpha)\vert \alpha^{\frac12}  \, \dx \alpha
\\
\notag
&
\hskip1cm
+
N^{\frac12}\int_{\frac{1}{H}}^{\frac{B}{H}} \vert V_{\ell_1}(\alpha) - T_{\ell_1}(\alpha)\vert\alpha^{-\frac12}  \, \dx \alpha.  
\end{align}
By  Lemma \ref{App-BCP-Gallagher-2} we obtain
\begin{align}
\notag
E_4&\ll
 HN^{-\frac12}  \Bigl(\int_{-\frac{1}{N}}^{\frac{1}{N}} \vert V_{\ell_1}(\alpha) - T_{\ell_1}(\alpha)\vert^2 \, \dx \alpha\Bigr)^{\frac{1}{2}}
 \\
\notag
& \hskip1cm
+
H N^{\frac12} 
\Bigl( \int_{\frac{1}{N}}^{\frac{1}{H}} \vert V_{\ell_1}(\alpha) - T_{\ell_1}(\alpha)\vert^2   \, \dx \alpha
\Bigr)^{\frac{1}{2}}
\Bigl( \int_{\frac{1}{N}}^{\frac{1}{H}} 
 \alpha \, \dx \alpha
\Bigr)^{\frac{1}{2}}
\\
\notag
&
\hskip1cm
+
N^{\frac12}\Bigl( \int_{\frac{1}{H}}^{\frac{B}{H}} \vert V_{\ell_1}(\alpha) - T_{\ell_1}(\alpha)\vert^2   \, \dx \alpha
\Bigr)^{\frac{1}{2}}
\Bigl( \int_{\frac{1}{H}}^{\frac{B}{H}} 
 \frac{\dx \alpha}{\alpha}
\Bigr)^{\frac{1}{2}}
 \\
 &
 \ll_{\ell_1,\ell_2}
 \label{E4-estim}
 N^{  \frac{1}{\ell_1} } A(N, -C),
\end{align}
for a suitable choice of $C=C(\eps)>0$, 
provided that $N^{-1+\frac{\eps}{2}}<B/H<N^{-1+\frac{5}{6\ell_1}-\eps}$; 
hence   $N^{1-\frac{5}{6\ell_1}+\eps} B \le H \le N^{1-\eps}$
suffices.

The estimate of $E_5$ runs analogously to the one of $E_4$.
We obtain
\begin{equation}
 \label{E5-estim}
E_5 
 \ll_{\ell_1,\ell_2}
 N^{  \frac{1}{\ell_2} } A(N, -C),
\end{equation}
for a suitable choice of $C=C(\eps)>0$, 
provided that $N^{-1+\frac{\eps}{2}}<B/H<N^{-1+\frac{5}{6\ell_2}-\eps}$; 
hence   $N^{1-\frac{5}{6\ell_2}+\eps} B \le H \le N^{1-\eps}$
suffices.
Moreover  by  \eqref{UH-estim}
we get
\begin{align}
  \label{E6-estim}
E_6
 \ll 
H  \int_{-\frac{1}{N}}^{\frac{1}{N}}   \, \dx \alpha
+
H N  \int_{\frac{1}{N}}^{\frac{1}{H}}   \alpha   \, \dx \alpha 
+
N \int_{\frac{1}{H}}^{\frac{B}{H}}    \, \dx \alpha  
\ll
\frac{NB}{H}.
\end{align}
Hence by \eqref{density-def-c-def} and \eqref{I4-split}-\eqref{E6-estim} we obtain for $\ell_1\ge 2$ that
 \begin{equation}
\label{I4-estim}
\I_4 \ll_{\ell_1,\ell_2} 
 H N^{\lambda-1} A(N, -C),
\end{equation}
for a suitable choice of $C=C(\eps)>0$, 
provided that $N^{-1+\frac{\eps}{2}}<B/H<N^{-1+\frac{5}{6\ell_2}-\eps}$; 
hence   $N^{1-\frac{5}{6\ell_2}+\eps} B\le H \le N^{1-\eps}$
suffices.

\subsection{Final words}

Summarizing, recalling that $2\le\ell_1\le \ell_2$, 
by \eqref{B-def}, \eqref{main-dissection}-\eqref{main-term}, 
\eqref{I2-estim}-\eqref{I3-estim} and \eqref{I4-estim},
we have that there exists $C=C(\eps)>0$ such that 
\begin{align*}
\sum_{n=N+1}^{N+H} &
r^{\second}_{\ell_1,\ell_2}(n)
=
c(\ell_1,\ell_2)   H N^{\lambda-1} 
+
\Odipg{\ell_1,\ell_2}{H N^{\lambda-1} A(N, -C)},
\end{align*}
uniformly for $N^{2-\frac{11}{6\ell_2}-\frac{1}{\ell_1}+\eps}\le H \le N^{1-\eps}$ 
which is  non-trivial only  for  $\ell_1=2, \ell_2\in\{2,3\}$.
Theorem \ref{thm-uncond} follows.

\section{Proof of Theorem \ref{thm-RH}}
{}From now on we assume the Riemann Hypothesis holds. Recalling \eqref{r-def},
we have  
\[
\sum_{n=N+1}^{N+H} e^{-n/N} 
R^{\second}_{\ell_1,\ell_2}(n)
=  
\int_{-\frac{1}{2}}^{\frac{1}{2}} \Vtilde_{\ell_1}(\alpha) \Vtilde_{\ell_2}(\alpha)  U(-\alpha,H)e(-N\alpha) \, \dx \alpha .
\]
Hence
\begin{align}
 \notag
&\sum_{n=N+1}^{N+H} e^{-n/N} 
R^{\second}_{\ell_1,\ell_2}(n)
 = \frac{\Gamma(1/\ell_1)\Gamma(1/\ell_2)}{\ell_1\ell_2}
 \int_{-\frac{1}{2}}^{\frac{1}{2}} z^{-\frac{1}{\ell_1}-\frac{1}{\ell_2}}  U(-\alpha,H)e(-N\alpha) \, \dx \alpha 
\\
 \notag
 &
\hskip0.5cm
+
\frac{\Gamma(1/\ell_1)}{\ell_1}
\int_{-\frac{1}{2}}^{\frac{1}{2}} z^{-\frac{1}{\ell_1}} \Bigl(\Vtilde_{\ell_2}(\alpha)- \frac{\Gamma(1/\ell_2)}{\ell_2 z^{\frac{1}{\ell_2}}} \Bigr) 
U(-\alpha,H)e(-N\alpha) \, \dx \alpha 
\\
 \notag
 &
\hskip0.5cm
+
\frac{\Gamma(1/\ell_2)}{\ell_2}
\int_{-\frac{1}{2}}^{\frac{1}{2}}  z^{-\frac{1}{\ell_2}} \Bigl(\Vtilde_{\ell_1}(\alpha)- \frac{\Gamma(1/\ell_1)}{\ell_1 z^{\frac{1}{\ell_1}}}\Bigr)  U(-\alpha,H)e(-N\alpha) \, \dx \alpha 
\\
 \notag
 &
\hskip0.5cm
+
\int_{-\frac{1}{2}}^{\frac{1}{2}}  \Bigl(\Vtilde_{\ell_1}(\alpha)- \frac{\Gamma(1/\ell_1)}{\ell_1 z^{\frac{1}{\ell_1}}}\Bigr) 
\Bigl(\Vtilde_{\ell_2}(\alpha)- \frac{\Gamma(1/\ell_2)}{\ell_2 z^{\frac{1}{\ell_2}}} \Bigr) 
 U(-\alpha,H)e(-N\alpha) \, \dx \alpha 
\\
 \label{main-dissection-RH-series}
 & 
 \hskip1cm
= \J_1 +\J_2 + \J_3 + \J_4,
 \end{align}
 say.
 Now we evaluate these terms.
 
\subsection{Computation of $\J_1$}

By Lemma \ref{Laplace-formula},
 \eqref{density-def-c-def}
and using  $e^{-n/N}=e^{-1}+ \Odi{H/N}$ for $n\in[N+1,N+H]$, $1\le H \le N$,
a direct calculation  gives
\begin{align}
\notag
\J_1
&
=
c(\ell_1,\ell_2)
\sum_{n=N+1}^{N+H} e^{-n/N} n^{\lambda-1}
+\Odipg{\ell_1,\ell_2}{\frac{H}{N}}
\\
\notag
&
=
\frac{c(\ell_1,\ell_2)}{e}
\sum_{n=N+1}^{N+H} n^{\lambda-1}
+\Odipg{\ell_1,\ell_2}{\frac{H}{N}+H^2N^{\lambda-2}}
\\
\label{J1-eval-RH-series}
&
=
c(\ell_1,\ell_2)
\frac{HN^{\lambda-1}}{e}
+\Odipg{\ell_1,\ell_2}{\frac{H}{N}+H^2N^{\lambda-2}+N^{\lambda-1}}.
 \end{align}

\subsection{Estimate of $\J_2$} %
{}From now on, we denote 
\begin{equation}
\label{Etilde-def}
\Etilde_{\ell}(\alpha) : =\Stilde_\ell(\alpha) - \frac{\Gamma(1/\ell)}{\ell z^{\frac{1}{\ell}}}.
\end{equation}
Using Lemma \ref{tilde-trivial-lemma}
we remark that
\begin{equation}
\label{V-tilde-approx}
\Bigl\vert
\Vtilde_{\ell}(\alpha)- \frac{\Gamma(1/\ell)}{\ell z^{\frac{1}{\ell}}}
\Bigr\vert
\le
\vert
\Etilde_{\ell}(\alpha)
\vert
+
\vert
\Vtilde_{\ell}(\alpha)-\Stilde_{\ell}(\alpha)
\vert
=
\vert
\Etilde_{\ell}(\alpha)
\vert
+\Odipm{\ell}{N^{\frac{1}{2\ell}}}.
\end{equation}

Hence
\begin{align}
\notag
\J_2
&\ll_{\ell_1,\ell_2}
\int_{-\frac{1}{2}}^{\frac{1}{2}} 
\vert z\vert^{-\frac{1}{\ell_1}}
\vert \Etilde_{\ell_2}(\alpha) \vert
\vert 
U(-\alpha,H)
\vert \, \dx \alpha 
\\&
\label{J2-RH-series-first-estim}
\hskip1cm
+
N^{\frac{1}{2\ell_2}}
\int_{-\frac{1}{2}}^{\frac{1}{2}} 
\vert z\vert^{-\frac{1}{\ell_1}}
\vert 
U(-\alpha,H)
\vert \, \dx \alpha 
=\A+\B,
\end{align}
say.
By \eqref{UH-estim} and \eqref{z-estim}
we have
\begin{align}
\notag
\B &\ll_{\ell_1,\ell_2}
H N^{\frac{1}{\ell_1}+\frac{1}{2\ell_2}-1} 
+
HN^{\frac{1}{2\ell_2}}
\int_{\frac{1}{N}}^{\frac{1}{H}} 
\alpha^{-\frac{1}{\ell_1}}
\, \dx \alpha 
+
N^{\frac{1}{2\ell_2}}
\int_{\frac{1}{H}}^{\frac{1}{2}} 
\alpha^{-\frac{1}{\ell_1}-1}
\, \dx \alpha 
\\
\label{B-estim}
&
\ll_{\ell_1,\ell_2}
H N^{\frac{1}{\ell_1}+\frac{1}{2\ell_2}-1} 
+
H^{\frac{1}{\ell_1}}N^{\frac{1}{2\ell_2}}.
\end{align}
By \eqref{UH-estim}, \eqref{z-estim},
the Cauchy-Schwarz inequality, 
Lemma \ref{LP-Lemma-gen}
and a partial integration argument
similar to the one used in the proof of Lemma \ref{partial-int-average}
(see the estimate of $M_2$ there),
we have
\begin{align}
\notag
\A &\ll_{\ell_1,\ell_2}
H N^{\frac{1}{\ell_1}}
\int_{-\frac{1}{N}}^{\frac{1}{N}} 
\vert \Etilde_{\ell_2}(\alpha) \vert
\, \dx \alpha 
+
H
\int_{\frac{1}{N}}^{\frac{1}{H}} 
\frac{\vert \Etilde_{\ell_2}(\alpha) \vert}
{\alpha^{\frac{1}{\ell_1}}}
\, \dx \alpha 
+
\int_{\frac{1}{H}}^{\frac{1}{2}} 
\frac{\vert \Etilde_{\ell_2}(\alpha) \vert}
{\alpha^{\frac{1}{\ell_1}+1}}
\, \dx \alpha 
\\
\notag
&\ll_{\ell_1,\ell_2}
H N^{\frac{1}{\ell_1}-\frac{1}{2}}
\Bigl(
\int_{-\frac{1}{N}}^{\frac{1}{N}} 
\vert \Etilde_{\ell_2}(\alpha) \vert^2
\, \dx \alpha 
\Bigr)^{\frac{1}{2}}
\\
\notag
&\hskip1cm
+
H
\Bigl(
\int_{\frac{1}{N}}^{\frac{1}{H}} 
\vert \Etilde_{\ell_2}(\alpha) \vert^2
\, \dx \alpha 
\Bigr)^{\frac{1}{2}}
\Bigl(
\int_{\frac{1}{N}}^{\frac{1}{H}} 
\frac{\dx \alpha} {\alpha^{\frac{2}{\ell_1}}}
\Bigr)^{\frac{1}{2}}
\\
\notag
&\hskip1cm+
\Bigl(
\int_{\frac{1}{H}}^{\frac{1}{2}} 
\vert \Etilde_{\ell_2}(\alpha) \vert^2
\frac{\dx \alpha} {\alpha^{2}}
\Bigr)^{\frac{1}{2}}
\Bigl(
\int_{\frac{1}{H}}^{\frac{1}{2}} 
\frac{\dx \alpha} {\alpha^{\frac{2}{\ell_1}}}
\Bigr)^{\frac{1}{2}}
\\
\label{A-estim}
&\ll_{\ell_1,\ell_2}
H N^{\frac{1}{\ell_1}+\frac{1}{2\ell_2}-1}L
+
H^{\frac{1}{\ell_1}} N^{\frac{1}{2\ell_2}}L^{\frac32}.
\end{align}
By \eqref{J2-RH-series-first-estim}-\eqref{A-estim}
we have
\begin{equation}
\J_2\ll_{\ell_1,\ell_2}
 \label{J2-estim-RH-series}  
 H N^{\frac{1}{\ell_1}+\frac{1}{2\ell_2}-1}L
+
H^{\frac{1}{\ell_1}} N^{\frac{1}{2\ell_2}}L^{\frac32}
\ll_{\ell_1,\ell_2}
H^{\frac{1}{\ell_1}} N^{\frac{1}{2\ell_2}}L^{\frac32}.
\end{equation}

\subsection{Estimate of $\J_3$}

The estimate of $\J_3$
is very similar to $\J_2$'s;
we just need to interchange $\ell_1$ with $\ell_2$.
We obtain
 \begin{equation}
\label{J3-estim-RH-series}
\J_3 \ll_{\ell_1,\ell_2}
 H N^{\frac{1}{2\ell_1}+\frac{1}{\ell_2}-1}L
+
H^{\frac{1}{\ell_2}} N^{\frac{1}{2\ell_1}}L^{\frac32}
\ll_{\ell_1,\ell_2}
H^{\frac{1}{\ell_2}} N^{\frac{1}{2\ell_1}}L^{\frac32}.
\end{equation}

\subsection{Estimate of $\J_4$}
 
 Using \eqref{density-def-c-def} and \eqref{V-tilde-approx} we get
 \begin{align}
 \notag
\I_4
& \ll_{\ell_1,\ell_2}
\int_{-\frac12}^{\frac12}
\vert \Etilde_{\ell_1}(\alpha) \vert \vert \Etilde_{\ell_2}(\alpha) \vert\vert U(-\alpha,H)\vert \, \dx \alpha
\\
\notag
&\hskip1cm
+
N^{\frac{1}{2\ell_2}}
\int_{-\frac12}^{\frac12}
\vert  \Etilde_{\ell_1}(\alpha) \vert  \vert U(-\alpha,H) \vert \, \dx \alpha
\\
\notag
&\hskip1cm
+
N^{\frac{1}{2\ell_1}}
\int_{-\frac12}^{\frac12}
\vert  \Etilde_{\ell_2}(\alpha) \vert  \vert U(-\alpha,H) \vert \, \dx \alpha
+
N^{\frac{\lambda}{2}}
\int_{-\frac12}^{\frac12}
  \vert U(-\alpha,H) \vert \, \dx \alpha
\\
\label{J4-split}&
= \E_1+\E_2+\E_3+\E_4,
 \end{align}
 say. 
 By the Cauchy-Schwarz inequality, \eqref{density-def-c-def},
 \eqref{UH-estim} and
Lemma  \ref{partial-int-average} 
we obtain
 \begin{align}
\notag
 \E_1
& \ll_{\ell_1,\ell_2}
\Bigl(
\int_{-\frac12}^{\frac12}
\vert  \Etilde_{\ell_1}(\alpha) \vert^2  \vert U(-\alpha,H) \vert \, \dx \alpha
\Bigr)^{\frac{1}{2}}
\Bigl(
\int_{-\frac12}^{\frac12}
\vert  \Etilde_{\ell_2}(\alpha) \vert^2  \vert U(-\alpha,H) \vert \, \dx \alpha
\Bigr)^{\frac{1}{2}}
\\&
 \label{E1-estim-RH-series}
\ll_{\ell_1,\ell_2} 
  N^{\frac{\lambda}{2}} L^3.
\end{align}
 By the Cauchy-Schwarz inequality,  \eqref{density-def-c-def},
Lemmas \ref{UH-average} and \ref{partial-int-average} 
we obtain
 \begin{align}
\notag
 \E_2 
& \ll_{\ell_1,\ell_2}
N^{\frac{1}{2\ell_2}}
\Bigl(
\int_{-\frac12}^{\frac12}
\vert  \Etilde_{\ell_1}(\alpha) \vert^2  \vert U(-\alpha,H) \vert \, \dx \alpha
\Bigr)^{\frac{1}{2}}
\Bigl(
\int_{-\frac12}^{\frac12}
\vert U(-\alpha,H) \vert \, \dx \alpha
\Bigr)^{\frac{1}{2}}
\\&
 \label{E2-estim-RH-series}
\ll_{\ell_1,\ell_2} 
  N^{\frac{\lambda}{2}} L^2.
\end{align}
 By the Cauchy-Schwarz inequality,  \eqref{density-def-c-def},
Lemmas \ref{UH-average} and \ref{partial-int-average} 
 we obtain
 \begin{align}
\notag
 \E_3 
& \ll_{\ell_1,\ell_2}
N^{\frac{1}{2\ell_1}}
\Bigl(
\int_{-\frac12}^{\frac12}
\vert  \Etilde_{\ell_2}(\alpha) \vert^2  \vert U(-\alpha,H) \vert \, \dx \alpha
\Bigr)^{\frac{1}{2}}
\Bigl(
\int_{-\frac12}^{\frac12}
\vert U(-\alpha,H) \vert \, \dx \alpha
\Bigr)^{\frac{1}{2}}
\\&
 \label{E3-estim-RH-series}
\ll_{\ell_1,\ell_2} 
  N^{\frac{\lambda}{2}} L^2.
\end{align}
 By \eqref{UH-estim} we immediately have
 \begin{equation}
 \label{E4-estim-RH-series}
 \E_4 
 \ll_{\ell_1,\ell_2}
  N^{\frac{\lambda}{2}} L.
 \end{equation}
Hence by \eqref{J4-split}-\eqref{E4-estim-RH-series}
we finally can write that
 \begin{equation}
\label{J4-estim-RH-series}
\J_4 \ll_{\ell_1,\ell_2}
 N^{\frac{\lambda}{2}} L^3.
 \end{equation}

\subsection{Final words}

Summarizing, recalling $2\le\ell_1\le \ell_2$, 
by  \eqref{density-def-c-def}, \eqref{main-dissection-RH-series}-\eqref{J1-eval-RH-series}, 
\eqref{J2-estim-RH-series}-\eqref{J3-estim-RH-series} and \eqref{J4-estim-RH-series},
we have
\begin{align}
\notag
\sum_{n=N+1}^{N+H} 
e^{-n/N}
R^{\second}_{\ell_1,\ell_2}(n)
&=
c(\ell_1,\ell_2)
\frac{HN^{\lambda-1}}{e}
\\&
\label{almost-done}
+\Odipg{\ell_1,\ell_2}{
\frac{H}{N}
+
H^2N^{\lambda-2}
+
H^{\frac{1}{\ell_1}} N^{\frac{1}{2\ell_2}}L^{\frac32}
}
\end{align}
which is an asymptotic formula 
for $\infty(N^{1-a(\ell_1,\ell_2)}L^{b(\ell_1)})\le H \le \odi{N}$,
where $a(\ell_1,\ell_2)$ and $b(\ell_1)$ are defined in \eqref{aB-def}.
{}From  $e^{-n/N}=e^{-1}+ \Odi{H/N}$ for $n\in[N+1,N+H]$, $1\le H \le N$,
we get 
\begin{align*}
   \sum_{n = N+1}^{N + H} 
R^{\second}_{\ell_1,\ell_2}(n)
&=
c(\ell_1,\ell_2)HN^{\lambda-1}
+\Odipg{\ell_1,\ell_2}{
H^2N^{\lambda-2}
+
H^{\frac{1}{\ell_1}} N^{\frac{1}{2\ell_2}}L^{\frac32}
}  
\\&
+
  \Odig{\frac{H}{N}\sum_{n = N+1}^{N + H} R^{\second}_{\ell_1,\ell_2}(n)
}.
\end{align*}
Using $e^{n/N}\le  e^{2}$ 
and \eqref{almost-done},
 the last error term is
$\ll_{\ell_1,\ell_2} H^2N^{\lambda-2}$.
Hence we get
\begin{equation*}
   \sum_{n = N+1}^{N + H} 
R^{\second}_{\ell_1,\ell_2}(n)
=
c(\ell_1,\ell_2)HN^{\lambda-1}
+\Odipg{\ell_1,\ell_2}{
H^2N^{\lambda-2}
+
H^{\frac{1}{\ell_1}} N^{\frac{1}{2\ell_2}}L^{\frac32}
},
\end{equation*}
uniformly for every $2\le \ell_1\le \ell_2$ and 
 $\infty(N^{1-a(\ell_1,\ell_2)}L^{b(\ell_1,\ell_2)})\le H \le\odi{N}$,
where $a(\ell_1,\ell_2)$ and $b(\ell_1)$ are defined in \eqref{aB-def}.
Theorem \ref{thm-RH} follows.

 \section{Proof of Theorem \ref{thm-uncond-HL}}
 
Assume $H>2B$ and $\ell_1,\ell_2\ge 2$; we'll see at the end
of the proof how the conditions in the statement of this theorem
follow; remark that in this case we cannot interchange 
the role of  $\ell_1,\ell_2$. We have 
\begin{align}
\notag
\sum_{n=N+1}^{N+H} &
 r^{\prime}_{\ell_1,\ell_2}(n)= 
\int_{-\frac12}^{\frac12} V_{\ell_1}(\alpha)  T_{\ell_2}(\alpha)  U(-\alpha,H)e(-N\alpha) \, \dx \alpha
\\
\notag
&=
\int_{-\frac{B}{H}}^{\frac{B}{H}} V_{\ell_1}(\alpha) T_{\ell_2}(\alpha)  U(-\alpha,H)e(-N\alpha) \, \dx \alpha 
\\&
\label{dissect-HL}
 + \!\!
\int\limits_{I(B,H)} \!\!\!\!\!  
V_{\ell_1}(\alpha) T_{\ell_2}(\alpha)  U(-\alpha,H)e(-N\alpha) \, \dx \alpha,
\end{align}
where  $I(B,H):=[-1/2,-B/H]\cup  [B/H, 1/2]$.
By the Cauchy-Schwarz inequality we have
\begin{align}
\notag
\int\limits_{I(B,H)} &
V_{\ell_1}(\alpha)  T_{\ell_2}(\alpha)  U(-\alpha,H)e(-N\alpha) \, \dx \alpha
\\
\label{CS-estim}
& \ll
\Bigl(
\int\limits_{I(B,H)} \!\!\! 
\vert V_{\ell_1}(\alpha) \vert ^{2} \vert U(-\alpha,H) \vert \, \dx \alpha
\Bigr)^{\frac12}
\Bigl(
\int\limits_{I(B,H)} \!\!\! 
\vert T_{\ell_2}(\alpha) \vert ^{2} \vert U(-\alpha,H) \vert \, \dx \alpha
\Bigr)^{\frac12}.
\end{align}

A similar computation to the one in \eqref{V-average} leads to 
\begin{align}
\notag
\int\limits_{I(B,H)} \!\!\! 
\vert T_{\ell}(\alpha) \vert ^{2} &\vert U(-\alpha,H) \vert \, \dx \alpha
\ll
\int_{\frac{B}{H}}^{\frac12} 
\vert T_{\ell}(\alpha) \vert ^{2} \frac{\dx \alpha}{\alpha}
 \\ 
 \label{T-average}
 &
 \ll_\ell
 N^{\frac{1}{\ell}}  
+
\frac{H L }{B} 
+
\int_{\frac{B}{H}}^{\frac12}  
(\xi N^{\frac{1}{\ell}}   + L )\ \frac{\dx \xi}{\xi^2}
 \ll_\ell
  N^{\frac{1}{\ell}} L 
+
\frac{H L }{B},
\end{align}
for every $\ell\ge 2$.
Hence, by  \eqref{CS-estim}-\eqref{T-average} and recalling \eqref{B-def} and  \eqref{V-average},   
we obtain
\begin{align}
\notag
 \int\limits_{I(B,H)} 
V_{\ell_1}(\alpha) & T_{\ell_2}(\alpha)  U(-\alpha,H) e(-N\alpha) \, \dx \alpha
 \\&
  \label{V-average-coda-HL}
 \ll_{\ell_1,\ell_2}
  N^{\frac{\lambda}{2}} L^{\frac32}
+
 \frac{ H^{\frac12} N^{\frac{1}{2\ell_1}}  L^{\frac32}}{B^{\frac{1}{2}}}
 +\frac{H L^{\frac32}}{B}
 \ll_{\ell_1,\ell_2}
\frac{H L^{\frac32}}{B}.
\end{align}

By \eqref{dissect-HL} and \eqref{V-average-coda-HL}, we get
\begin{align}
 \notag
\sum_{n=N+1}^{N+H} 
&r^{\prime}_{\ell_1,\ell_2}(n)
 =  
\int_{-\frac{B}{H}}^{\frac{B}{H}} V_{\ell_1}(\alpha) T_{\ell_2}(\alpha)  U(-\alpha,H)e(-N\alpha) \, \dx \alpha 
 +
\Odipg{\ell_1,\ell_2}{\frac{H L^{\frac32}}{B}}.
\end{align}

Hence
\begin{align} 
 \notag
\sum_{n=N+1}^{N+H} 
&r^{\prime}_{\ell_1,\ell_2}(n)  
 = \int_{-\frac{B}{H}}^{\frac{B}{H}} f_{\ell_1}(\alpha) f_{\ell_2}(\alpha)  U(-\alpha,H)e(-N\alpha) \, \dx \alpha 
 \\
 \notag
 &
\hskip0.5cm
+
\int_{-\frac{B}{H}}^{\frac{B}{H}} f_{\ell_2}(\alpha) (V_{\ell_1}(\alpha) - f_{\ell_1}(\alpha) ) U(-\alpha,H)e(-N\alpha) \, \dx \alpha 
\\
 \notag
 &
\hskip0.5cm
+
\int_{-\frac{B}{H}}^{\frac{B}{H}} f_{\ell_1}(\alpha) (T_{\ell_2}(\alpha) - f_{\ell_2}(\alpha) ) U(-\alpha,H)e(-N\alpha) \, \dx \alpha 
 \\
 \notag
 &
 \hskip0.5cm
 +
\int_{-\frac{B}{H}}^{\frac{B}{H}} (V_{\ell_1}(\alpha)-f_{\ell_1}(\alpha) ) (T_{\ell_2}(\alpha) - f_{\ell_2}(\alpha) ) U(-\alpha,H)e(-N\alpha) \, \dx \alpha
 \\
 \notag
 &
\hskip0.5cm
+
\Odipg{\ell_1,\ell_2}{\frac{H L^{\frac32}}{B}}
 \\
 \label{main-dissection-HL}
 & 
= \I_1 +\I_2 + \I_3 + \I_4 + E,
 \end{align}
 say. We now evaluate these terms.
The main term $\I_1$ can be evaluated as in \S\ref{comp-main-term};
by \eqref{I1-eval}-\eqref{main-term} it is 
\begin{equation}
\label{I1-eval-HL}
\I_1 
=
c(\ell_1,\ell_2)   H N^{\lambda-1} +
\Odipg{\ell_1,\ell_2}{ \Bigl(\frac{H}{B}\Bigr)^{\lambda}+ H N^{\lambda-1} A(N, -C)},
 \end{equation}
for a suitable choice of   $C=C(\eps)>0$.
$\I_2$ can be estimated as in \S\ref{estim-I2}; by \eqref{I2-estim} it is 
 \begin{equation}
\label{I2-estim-HL}
\I_2 \ll_{\ell_1,\ell_2}
 H N^{\lambda-1}  A(N, -C),
\end{equation}
for a suitable choice of   $C=C(\eps)>0$, 
provided that $N^{-1+\frac{\eps}{2}}<B/H<N^{-1+\frac{5}{6\ell_1}-\eps}$; 
hence   $N^{1-\frac{5}{6\ell_1}+\eps} B \le H \le N^{1-\eps}$
suffices.

 \subsection{Estimate of $\I_3$}
Using \eqref{f-ell-T-ell-estim}  we obtain that
 \begin{align*}
\notag
\I_3
 \ll 
   \int_{-\frac{B}{H}}^{\frac{B}{H}} 
\vert  f_{\ell_1}(\alpha) \vert  (1+\vert \alpha \vert N)^{\frac12} \vert U(-\alpha,H) \vert \, \dx \alpha 
 \end{align*}
and the right hand side is equal to $E_2$ of  \S\ref{estim-I2}; hence  
by \eqref{E2-estim} we have 
 \begin{equation}
\label{I3-estim-HL}
\I_3 \ll_{\ell_1,\ell_2}
  N^{1/\ell_1}  A^{1/2-1/\ell_1} L^{1/2} (\log A)^{1/2},
\end{equation}
where $A$ is defined in \eqref{A-def}.

\subsection{Estimate of $\I_4$}
\label{estim-I4-HL}
By \eqref{main-dissection} and  \eqref{V-approx} we can write
 \begin{align}
 \notag
\I_4
& \ll_{\ell_1,\ell_2}
\int_{-\frac{B}{H}}^{\frac{B}{H}}
\vert  V_{\ell_1}(\alpha) -T_{\ell_1}(\alpha)\vert  (1+\vert \alpha \vert N)^{\frac12} \vert U(-\alpha,H) \vert \, \dx \alpha
\\
\label{I4-split-HL}
&\hskip1cm
+
\int_{-\frac{B}{H}}^{\frac{B}{H}}
 (1+\vert \alpha \vert N)  \vert U(-\alpha,H) \vert \, \dx \alpha
= R_1+R_2,
 \end{align}
 say.
$R_1$ is equal to $E_4$ of \S\ref{estim-I4}; hence
we have
\begin{equation} 
 \label{R1-estim}
R_1
 \ll_{\ell_1,\ell_2}
 N^{  \frac{1}{\ell_1} } A(N, -C),
\end{equation}
for a suitable choice of   $C=C(\eps)>0$, 
provided that $N^{-1+\frac{\eps}{2}}<B/H<N^{-1+\frac{5}{6\ell_1}-\eps}$; 
hence   $N^{1-\frac{5}{6\ell_1}+\eps} B \le H \le N^{1-\eps}$
suffices. 
$R_2$ is equal to $E_6$ of \S\ref{estim-I4}; hence 
we get
\begin{equation}
  \label{R2-estim}
R_2 
\ll
\frac{NB}{H}.
\end{equation} 
Summarizing, by \eqref{density-def-c-def} and \eqref{I4-split-HL}-\eqref{R2-estim}, we obtain
 \begin{equation}
\label{I4-estim-HL}
\I_4 \ll_{\ell_1,\ell_2} 
 H N^{\lambda-1} A(N, -C),
\end{equation}
for a suitable choice of   $C=C(\eps)>0$, 
provided that $N^{-1+\frac{\eps}{2}}<B/H<N^{-1+\frac{5}{6\ell_1}-\eps}$; 
hence   $N^{1-\frac{5}{6\ell_1}+\eps} B\le H \le N^{1-\eps}$
suffices.

\subsection{Final words} 

Summarizing, recalling that $\ell_1, \ell_2\ge 2$, 
by \eqref{B-def}, \eqref{main-dissection-HL}-\eqref{I3-estim-HL} and \eqref{I4-estim-HL},
we have that there exists $C=C(\eps)>0$ such that 
\begin{align*}
\sum_{n=N+1}^{N+H} &
r^{\prime}_{\ell_1,\ell_2}(n)
=
c(\ell_1,\ell_2)   H N^{\lambda-1} 
+
\Odipg{\ell_1,\ell_2}{H N^{\lambda-1} A(N, -C)},
\end{align*}
uniformly for    $N^{2-\frac{11}{6\ell_1} -\frac{1}{\ell_2}+\eps}\le H \le N^{1-\eps}$  which is  non-trivial  only  for $\ell_1=2$ and $2\le \ell_2\le 11$, or $\ell_1=3$ and $ \ell_2=2$.
Theorem \ref{thm-uncond-HL} follows.

\section{Proof of Theorem \ref{thm-RH-HL} }  
In this section we need some additional definitions and lemmas.
Letting
\begin{equation}
\label{omega-def}
\omega_{\ell}(\alpha)
=
\sum_{m=1}^{\infty}   
e^{-m^{\ell}/N} e(m^{\ell}\alpha)
=
\sum_{m=1}^{\infty}   
e^{-m^{\ell}z},
\end{equation}
we have the following
\begin{lemma}[Lemma 2 of \cite{LanguascoZ2016b}] 
\label{zac-lemma-series}
Let $\ell\ge 2 $ be an integer and $0<\xi\le 1/2$. Then
\[
\int_{-\xi}^{\xi} 
|\omega_{\ell}(\alpha)|^2 \ \dx\alpha 
\ll_{\ell}
  \xi N^{\frac{1}{\ell}}
  + 
\begin{cases}
L & \text{if}\ \ell =2\\
1 & \text{if}\ \ell > 2
\end{cases}
\]
\ \textrm{and} 
\[
\int_{-\xi}^{\xi} 
|\Stilde_{\ell}(\alpha)|^2 \ \dx\alpha 
\ll_{\ell}
\xi N^{\frac{1}{\ell}} L  +
\begin{cases}
L^{2} & \text{if}\ \ell =2\\
1 & \text{if}\ \ell > 2.
\end{cases}
\]
\end{lemma}
Recalling the definition of the $\theta$-function
\[
 \theta(z)
=
\sum_{n=-\infty}^{\infty}   
e^{-n^{2}/N} e(n^{2}\alpha)
=
\sum_{n=-\infty}^{\infty}   
e^{-n^{2}z} 
=
1+2\omega_{2}(\alpha),
\]
its modular relation
(see, \emph{e.g.}, Proposition VI.4.3, page 340, of Freitag and Busam
\cite{FreitagB2009}) gives  that
\(
\theta(z)
=
(\pi/z)^{\frac{1}{2}} 
\theta ( \pi^2/z)
\)
for $\Re(z)>0$.
Hence  we have
\begin{equation}
\label{omega-approx}
\omega_{2}(\alpha)  
= 
\frac12
\left(\frac{\pi}{z}\right)^{\frac{1}{2}} 
\!\!\!\!
- \frac12
+
\left(\frac{\pi}{z}\right)^{\frac{1}{2}} 
\sum_{j=1}^{+\infty}   e^{-j^{2}\pi^{2}/z},
\quad
\text{for}
\ \Re(z)>0.
\end{equation} 
For the series in \eqref{omega-approx} we have 
\begin{lemma} [Lemma 4 of \cite{LanguascoZ2016b}]
\label{omega-Y}
Let $N$ be a large integer,  $z= 1/N-2\pi i\alpha$, $\alpha\in [-1/2,1/2]$ and 
$Y=\Re(1/z)>0$.
We have 
\[
\Bigl\vert 
\sum_{j=1}^{+\infty}   e^{-j^{2}\pi^{2}/z}
\Bigr\vert  
\ll 
\begin{Biggcases} 
e^{- \pi^{2}  Y } & \textrm{for} \ Y\ge 1 \\
Y^{-\frac{1}{2}} & \textrm{for} \ 0 <Y\le 1.
\end{Biggcases}
\]
\end{lemma} 
Since
\begin{equation}
\notag
Y
=
\Re(1/z)
=
\frac{N}{1+4\pi^2\alpha^2N^2}
\ge
\frac{1}{5\pi^2}
\begin{cases}
N & \textrm{if} \ \vert \alpha \vert \le 1/N\\
(\alpha^2N)^{-1} & \textrm{if} \ \vert \alpha \vert > 1/N,
\end{cases}
\end{equation}
from Lemma \ref{omega-Y} we get 
\begin{equation}
\label{omega-Y-estim}
\Bigl\vert 
\sum_{j=1}^{+\infty} e^{-j^{2}\pi^{2}/z}
\Bigr\vert  
\ll
\begin{cases}
e^{-  N/5} & \textrm{if} \ \vert \alpha \vert \le 1/N\\
\exp(- 1/(5\alpha^2N)) & \textrm{if} \ 1/N<\vert \alpha \vert = \odim{N^{-\frac{1}{2}}}\\
1+ N^{\frac{1}{2}} \vert \alpha\vert & \textrm{otherwise}.
\end{cases}
\end{equation}

 \begin{lemma}
\label{partial-int-average-omega2}
Let $N$ be a
sufficiently large integer and $L= \log N$.
  We have 
\[
\int_{-\frac{1}{2}}^{\frac{1}{2}}
      \vert \omega_{2}(\alpha)  \vert^{2}
        \vert U(\alpha,H)\vert 
     \, \dx \alpha
\ll
N^{\frac{1}{2}}  L + HL.
\] 
\end{lemma}

\begin{proof}
By \eqref{UH-estim} we have 
\begin{align}
\notag
\int_{-\frac{1}{2}}^{\frac{1}{2}}
      \vert \omega_{2}(\alpha)  \vert^{2}
        \vert U(\alpha,H)\vert 
     \, \dx \alpha
&\ll
H
\int_{-\frac1H}^{\frac1H} \,
\vert \omega_{2}(\alpha)  \vert^{2}
\ \dx \alpha
+
\int_{\frac1H}^{\frac12} \,
\vert \omega_{2}(\alpha)  \vert^{2}
\frac{\dx \alpha}{\alpha}
\\&
\notag
\hskip1cm
+
\int_{-\frac12}^{-\frac1H} \,
\vert \omega_{2}(\alpha)  \vert^{2}
\frac{\dx \alpha}{\alpha}
\\
\label{M-split-1}
&
=
M_1+M_2+M_3,
\end{align}
say. By Lemma \ref{zac-lemma-series} we immediately get that
\begin{equation}
\label{M1-estim-1}
M_1 \ll N^{\frac{1}{2}}  + HL.
\end{equation}
By a partial integration and Lemma \ref{zac-lemma-series} we obtain
\begin{align}
\notag
M_2 
&\ll 
\int_{-\frac12}^{\frac12} \,
\vert \omega_{2}(\alpha)  \vert^{2}
\ \dx \alpha
+
H \int_{-\frac1H}^{\frac1H} \,
\vert \omega_{2}(\alpha)  \vert^{2}
\ \dx \alpha 
+
\int_{\frac1H}^{\frac12}
\Bigl(
\int_{-\xi}^{\xi}
\vert \omega_{2}(\alpha)  \vert^{2}
\ \dx \alpha
\Bigr)
\frac{\ \dx \xi}{\xi^2}
\\&
\label{M2-estim-1}
\ll
N^{\frac{1}{2}} + HL 
+
\int_{\frac1H}^{\frac12} \frac{N^{\frac{1}{2}} \xi  + L }{\xi^2}\ \dx \xi
\ll
N^{\frac{1}{2}}  L + HL.
\end{align}
A similar computation leads to $M_3 \ll N^{\frac{1}{2}}  L + HL$.
By \eqref{M-split-1}-\eqref{M2-estim-1}, the lemma follows.
\end{proof}

{}From now on we assume the Riemann Hypothesis holds.
Let  $1<D=D(N)<H/2$ to be chosen later and 
$I(D,H):=[-1/2,-D/H]\cup  [D /H,1/2]$.
By \eqref{main-defs} 
and \eqref{omega-def}-\eqref{omega-approx},
and recalling \eqref{Etilde-def},
it is an easy matter to see that
\begin{align}
\notag
\sum_{n = N+1}^{N + H} & e^{-n/N}
R^{\prime}_{\ell,2}(n)
= 
\int_{-\frac{1}{2}}^{\frac{1}{2}}  \Vtilde_{\ell}(\alpha) \omega_{2}(\alpha)  U(-\alpha,H) e(-N\alpha) \, \dx \alpha
\\
\notag
 &=
  \int_{-\frac{1}{2}}^{\frac{1}{2}}  (\Vtilde_{\ell}(\alpha)- \Stilde_{\ell}(\alpha))\, \omega_{2}(\alpha)  U(-\alpha,H) e(-N\alpha) \, \dx \alpha
 \\
 \notag
 &  \hskip0.5cm
 + 
  \frac{\Gamma(1/\ell)}{2\ell}
  \int_{-\frac{D}{H}}^{\frac{D}{H}}   \Bigl( 
 \frac{\pi^{\frac{1}{2}}}{ z^{\frac{1}{2}+\frac{1}{\ell}} } -   \frac{1}{z^{\frac{1}{\ell}}}\Bigr)
 U(-\alpha,H) e(-N\alpha)\, \dx \alpha
 \\
 \notag
 & \hskip0.5cm
  +
 \int_{-\frac{D}{H}}^{\frac{D}{H}} 
 \Etilde_{\ell}(\alpha) 
 \omega_{2}(\alpha) U(-\alpha,H)
 e(-N\alpha)\, \dx \alpha 
\\
 \notag
 &  \hskip0.5cm
 +
 \frac{\pi^{\frac{1}{2}} \Gamma(1/\ell)}{\ell}
\int_{-\frac{D}{H}}^{\frac{D}{H}}   
 \frac{1}{z^{\frac{1}{2}+\frac{1}{\ell}} }
 \Bigl( \sum_{j=1}^{+\infty}   e^{-j^{2}\pi^{2}/z} \Bigr)
 U(-\alpha,H) e(-N\alpha)\, \dx \alpha 
\\
 \notag
 &  \hskip0.5cm+ 
\int\limits_{\mathclap{I(D,H)}} 
\Stilde_{\ell}(\alpha)\omega_{2}(\alpha)   U(-\alpha,H)  e(-N\alpha) \, \dx \alpha
\\
\label{approx-th1-part4} 
&
= I_{0} + I_{1}+I_{2}+I_{3} +I_{4}, 
\end{align}
say.
Using Lemma \ref{tilde-trivial-lemma}  and the  Cauchy-Schwarz inequality we have
\begin{align*} 
I_{0} 
\ll_\ell 
N^{\frac{1}{2\ell}}
\Bigl(
\int_{-\frac{1}{2}}^{\frac{1}{2}}
      \vert \omega_{2}(\alpha)  \vert^{2}
        \vert U(\alpha,H)\vert 
     \, \dx \alpha
     \Bigr)^{\frac{1}{2}}
 \Bigl(
      \int_{-\frac{1}{2}}^{\frac{1}{2}}
       \vert U(\alpha,H) \vert
        \, \dx \alpha        
\Bigr)^{\frac{1}{2}}.
\end{align*}
By Lemmas \ref{UH-average} and \ref{partial-int-average-omega2}
we obtain
\begin{equation}
\label{I0-estim-omega}
I_{0}
\ll_\ell 
N^{\frac{1}{2\ell}}
(N^{\frac{1}{2}}L  + H L)^{\frac{1}{2}}L^{\frac12}  
\ll
N^{\frac{1}{4}+\frac{1}{2\ell}}L+ H^{\frac{1}{2}}N^{\frac{1}{2\ell}}L.
\end{equation}

Now we evaluate $I_{1}$. Using Lemma \ref{Laplace-formula}, 
 \eqref{UH-estim} and   $e^{-n/N}=e^{-1}+ \Odi{H/N}$ for $n\in[N+1,N+H]$, $1\le H \le N$,
we immediately get 
\begin{align}
\notag
I_{1}
&=
\frac{\Gamma(1/\ell)}{2\ell}
 \sum_{n = N+1}^{N + H}   \Bigl( \frac{\pi^{\frac{1}{2}}}{\Gamma(\frac{1}{2}+\frac{1}{\ell})} 
 n^{\frac{1}{\ell}-\frac{1}{2}}  - n^{\frac{1}{\ell}-1}    \Bigr)  e^{-n/N} + \Odipg{\ell}{\frac{H}{N}}
 \\&
 \notag
 \hskip1cm
+
\Odipg{\ell}{\int_{\frac{D}{H}}^{\frac{1}{2}} \frac{\dx\alpha}{\alpha^2}}
 \\
 \label{I1-eval-part4}
 & =
    \frac{c(\ell,2)}{e} HN^{\frac{1}{\ell}-\frac{1}{2}}
 + \Odipg{\ell}{\frac{H}{N^{1-\frac{1}{\ell}}}+\frac{H^2}{N^{\frac{3}{2}-\frac{1}{\ell}}}+\frac{H}{D}}.
\end{align}
To have that the first term in $I_{1}$ dominates in $I_{0} + I_{1}$ we need that 
$D= \infty(N^{\frac{1}{2}-\frac{1}{\ell}})$, $H=\odi{N}$ and  $H=\infty(N^{1-\frac{1}{\ell}}L^{2})$,
$\ell\ge 2$.

Now we estimate $I_{3}$. Assuming $H=\infty(N^{\frac{1}{2}}D)$, 
by \eqref{z-estim} and \eqref{omega-Y-estim}, we have,
using the substitution $u=1/(5N\alpha^2)$ in the last integral, that
\begin{align}
\notag
I_{3}
&\ll_\ell 
\frac{HN^{\frac{1}{2}+\frac{1}{\ell}}}{e^{N/5}} \int_{-\frac{1}{N}}^{\frac{1}{N}}  \dx \alpha
+
\frac{H}{e^{H^2/(5N)}} \int_{\frac{1}{N}}^{\frac{1}{H}}  \frac{\dx \alpha}{\alpha^{\frac{1}{2}+\frac{1}{\ell}}} 
+
\int_{\frac{1}{H}}^{\frac{D}{H}}  \frac{\dx \alpha}{\alpha^{3/2+\frac{1}{\ell}}e^{1/(5N\alpha^2)}}  
\\
\notag
&
\ll_\ell 
\frac{HN^{\frac{1}{\ell}-\frac{1}{2}}}{e^{N/5}} + \frac{HN^{\frac{1}{\ell}-\frac{1}{2}}L}{e^{H^2/(5N)}} 
+
N^{\frac{1}{4}+\frac{1}{2\ell}}
\int_{H^2/(5ND^2)}^{H^2/(5N)} u^{-3/4+\frac{1}{2\ell}}e^{-u}\ \dx u
\\
\label{I3-estim-final-part4}
&
\ll_\ell 
\frac{HN^{\frac{1}{\ell}-\frac{1}{2}}L}{e^{H^2/(5N)}} 
+
N^{\frac{1}{4}+\frac{1}{2\ell}} = \odim{HN^{\frac{1}{\ell}-\frac{1}{2}}},
\end{align}
provided that  $H=\infty(N^{\frac{1}{2}} \log L)$ and $H=\infty(N^{1-\frac{1}{\ell}})$, $\ell\ge 2$.

Now we estimate $I_{2}$. Recalling that $H=\infty(N^{\frac{1}{2}}D)$,
for every $\vert \alpha \vert \le D/H$
we have, by \eqref{omega-approx}-\eqref{omega-Y-estim}, 
that $\vert\omega_{2}(\alpha) \vert \ll \vert z \vert ^{-\frac{1}{2}}$. 
Hence
\[
I_{2} 
\ll
 \int_{-\frac{D}{H}}^{\frac{D}{H}}
      \vert \Etilde_{\ell}(\alpha)  \vert  \frac{\vert U(\alpha,H)\vert}{\vert z \vert ^{\frac{1}{2}}}
     \, \dx \alpha.
\]
Using  \eqref{z-estim} and the Cauchy-Schwarz inequality we get
\begin{align}
\notag
I_{2}
&\ll 
H N^{\frac{1}{2}} \Bigl(  \int_{-\frac{1}{N}}^{\frac{1}{N}} \dx \alpha  \Bigr)^{\frac{1}{2}}
\Bigl(
 \int_{-\frac{1}{N}}^{\frac{1}{N}}  \vert \Etilde_{\ell}(\alpha)  \vert^{2}   \, \dx \alpha 
 \Bigr)^{\frac{1}{2}}
  \\
  \notag
     &\hskip1cm 
      +
 H \Bigl(  \int_{\frac{1}{N}}^{\frac{1}{H}}    \frac{\dx \alpha}{\alpha^{\frac{1}{2}}}\Bigr)^{\frac{1}{2}} 
\Bigl(
 \int_{\frac{1}{N}}^{\frac{1}{H}}\!\!  \vert \Etilde_{\ell}(\alpha)  \vert^{2}  
 \frac{\dx \alpha}{\alpha^{\frac{1}{2}}}
  \Bigr)^{\frac{1}{2}}
   \\
  \notag
     &\hskip1cm     +
  \Bigl(  \int_{\frac{1}{H}}^{\frac{D}{H}}   \frac{\dx \alpha}{\alpha^{\frac{3}{2}}}\Bigr)^{\frac{1}{2}} 
\Bigl(
 \int_{\frac{1}{H}}^{\frac{D}{H}}  \vert \Etilde_{\ell}(\alpha)  \vert^{2}   
 \frac{\dx \alpha}{\alpha^{\frac{3}{2}}}
   \Bigr)^{\frac{1}{2}}
 \end{align}
 
By  Lemma \ref{LP-Lemma-gen}  we get 
\begin{align}
\notag
I_{2}
& 
 \ll_{\ell}
H N^{\frac{1}{2\ell}-\frac{1}{2}} L 
+ 
H^{3/4}  N^{\frac{1}{2\ell}}  L  \Bigl(  \frac{1}{H^{\frac{1}{2}}}+ \int_{\frac{1}{N}}^{\frac{1}{H}}  \frac{\dx \xi}{\xi^{\frac{1}{2}}}   \Bigr)^{\frac{1}{2}} 
  \\
  \notag
     &\hskip1cm 
 +
H^{\frac{1}{4}} N^{\frac{1}{2\ell}} L \Bigl( H^{\frac{1}{2}} +   \int_{\frac{1}{H}}^{\frac{D}{H}}  \frac{\dx \xi}{\xi^{\frac{3}{2}}}   \Bigr)^{\frac{1}{2}}
 \\&
  \label{I2-estim-final-part4}
 \ll_{\ell}
 H^{\frac{1}{2}}N^{\frac{1}{2\ell}} L.
\end{align}

We remark that $I_2=\odim{HN^{\frac{1}{\ell}-\frac{1}{2}} }$ provided that $H=\infty(N^{1-\frac{1}{\ell}} L^{2})$,
$\ell\ge 2$.

Now we estimate $I_{4}$. By \eqref{UH-estim}, Lemma \ref{zac-lemma-series} 
and a partial integration argument
we get
\begin{align}
\notag
I_4
&\ll
\int_{\frac{D}{H}}^{\frac{1}{2}}
\vert  \Stilde_{\ell}(\alpha) \omega_{2}(\alpha) \vert
\frac{\dx \alpha}{\alpha}
\ll
\Bigl(
\int_{\frac{D}{H}}^{\frac{1}{2}}  
\vert \Stilde_{\ell}(\alpha)\vert^2 \frac{\dx \alpha}{\alpha}
\Bigr)^{\frac{1}{2}}
\Bigl(
\int_{\frac{D}{H}}^{\frac{1}{2}}  
\vert \omega_{2}(\alpha) \vert^2\frac{\dx \alpha}{\alpha}
\Bigr)^{\frac{1}{2}}
\\
\notag
&\ll_\ell
\Bigl(
 N^{\frac{1}{\ell}}   L
+
\frac{HL^{2}}{D}
+
L
\int_{\frac{D}{H}}^{\frac{1}{2}}  
(\xi N^{\frac{1}{\ell}}  + L) \frac{\dx \xi}{\xi^2}
\Bigr)^{\frac{1}{2}}
  \\
  \notag
     &\hskip1cm 
     \times
\Bigl(
 N^{\frac{1}{2}} 
+
\frac{HL}{D}
+
\int_{\frac{D}{H}}^{\frac{1}{2}}  
(\xi N^{\frac{1}{2}}  + L) \frac{\dx \xi}{\xi^2}
\Bigr)^{\frac{1}{2}}
\\
\label{I4-estim-final-part4}
&\ll_\ell
 L^{\frac32} \Bigl(
 N^{\frac{1}{4}+\frac{1}{2\ell}}
+
\frac{H}{D}
\Bigr) .
\end{align}
Clearly we have that $I_4=\odim{HN^{\frac{1}{\ell}-\frac{1}{2}}}$
provided that $D=\infty(N^{\frac{1}{2}-\frac{1}{\ell}}L^{\frac32})$ and $H=\infty(N^{\frac{3}{4}-\frac{1}{2\ell}}L^{\frac32})$, $\ell\ge 2$.

Combining the previous conditions on $H$ and $D$ we can choose $D=N^{\frac{1}{2}-\frac{1}{\ell}}L^{2}/(\log L)$
and 
$H=\infty(N^{1-\frac{1}{\ell}}L^{2})$.
Hence using \eqref{approx-th1-part4}-\eqref{I4-estim-final-part4} we  
can write  
\begin{align*}
\sum_{n = N+1}^{N + H}  e^{-n/N}
R^{\prime}_{\ell,2}(n)  
& =
\frac{ c(2,\ell) }{e} HN^{\frac{1}{\ell}-\frac{1}{2}}
\\&
  + \Odipg{\ell}{\frac{H^2}{N^{\frac{3}{2}-\frac{1}{\ell}}}+ \frac{H N^{\frac{1}{\ell}-\frac{1}{2}}\log L} {L^{\frac12}} 
+
H^{\frac{1}{2}}N^{\frac{1}{2\ell}} L}.
\end{align*}
Theorem \ref{thm-RH-HL}  follows for $\infty(N^{1-\frac{1}{\ell}}L^{2}) \le H \le \odi{N}$, $\ell\ge 2$,
since the exponential weight $e^{-n/N}$
can be removed as we did at the bottom of the proof of Theorem \ref{thm-RH}.
\qed

\end{document}